\documentclass[11pt]{article}  
\usepackage{amsthm}

\makeatletter
\newtheorem*{rep@theorem}{\rep@title}
\newcommand{\newreptheorem}[2]{%
\newenvironment{rep#1}[1]{%
 \def\rep@title{#2 \ref{##1}}%
 \begin{rep@theorem}}%
 {\end{rep@theorem}}}
\makeatother

\usepackage{amssymb, amsmath}
\usepackage{verbatim}
\newtheorem{theorem}{Theorem}
\newtheorem{definition}{Definition}
\newtheorem{lemma}{Lemma}
 \newtheorem{prop}{Proposition}
 \newtheorem{corollary}{Corollary}
 
 \newreptheorem{theorem}{Theorem}

\usepackage[margin=.5in,footskip=0.25in]{geometry}
\newgeometry{left=0.5in,right=0.5in,top=.6in,bottom=1in}
 \begin{document}

 \title{Restricted Carleson Variations at Endpoint and Discretized Hilbert Transforms in the Plane}
 \author{Robert M. Kesler}
 \maketitle

 \abstract{We provide elementary proofs that the 2-variation Carleson operator $V_2$ along with explicit bilinear multipliers adapted to $\{\xi_1 + \xi_2 = 0\}$ satisfy no $L^p$ estimates. Furthermore, we obtain $L^p \rightarrow L^p$ estimates when $2 < p <\infty$ for a smooth restricted variant of $V_2$ that is defined a priori  on Schwartz functions by the formula

\begin{eqnarray*}
\mathcal{V}^{res}_2 : f \mapsto \sup_{R \in \mathbb{R}_+} ~~\sup_{0 \leq \alpha  < R} ~~\left(\sum_{j \in \mathbb{Z}}  \left|f*\mathcal{F}^{-1} \left[ \tilde{1}_{[\alpha + j R, \alpha + (j+1)R]}\right] \right|^2 \right)^{1/2} 
\end{eqnarray*}
where $\tilde{1}_{I} (x) := \tilde{1}(|I|^{-1} (x-c_I))$ for all intervals $I = [c_I - |I|/2, c_I + |I|/2] \subset \mathbb{R}$ and $\tilde{1} \in C^\infty([-1/2, 1/2])$. We then study bi-sublinear variants of $\mathcal{V}_2^{res}$ before showing that multipliers, which are adapted to $\{\xi_1 + \xi_2=0\}$ and periodically discretized along each frequency scale, map $L^{p_1}(\mathbb{R}) \times L^{p_2}(\mathbb{R}) \rightarrow L^{p_1 p_2 / (p_1 + p_2)}(\mathbb{R})$ provided $2 \leq p_1, p_2 <\infty$ and $\frac{1}{p_1} + \frac{1}{p_2} <1$. 

\tableofcontents
 \section{Introduction}
R. Oberlin, A. Seeger, T. Tao, C. Thiele, and J. Wright  prove in \cite{MR2881301} that the r-variation Carleson operator defined for $r>0$ and $f \in \mathcal{S}(\mathbb{R})$ by the formula
 
 \begin{eqnarray*}
 V_r: f \mapsto \sup_{K \in \mathbb{N}} ~~\sup_{N_1 < N_2<... <N_K} \left(\sum_{j=1}^{K-1} \left| \int_{N_j}^{N_{j+1}} \hat{f}(\xi) e^{2 \pi i \xi x} d \xi \right|^r \right)^{1/r}
 \end{eqnarray*}
 extends to a continuous map of $L^p(\mathbb{R})$ into $L^p(\mathbb{R})$  for all $r^\prime<p<\infty$ and $r>2$.  With this notation, $V_\infty$ is the Carleson operator $C:L^p (\mathbb{R}) \rightarrow L^p(\mathbb{R})$ for $1 < p <\infty$ and $V_1(f)=||\hat{f}||_1 1_{\mathbb{R}}$ for all $f$ belonging to the Wiener algebra. That $r \geq 2$ is necessary for estimates to hold is immediate by routine arguments using Rademacher functions. Moreover, estimates at the variational endpoint $r=2$ are ruled out by applying probabilistic arguments of Qian in \cite{MR1640349} combined with variation estimates developed by Jones and Wang in \cite{MR2067131}. Our first result provides a simple and direct counterexample to the boundedness of $V_2$ and manages to say a bit more through the use of Gaussian chirps like those appearing in work of Muscalu, Tao, and Thiele  \cite{MR1981900}. Before mentioning the precise statement, we introduce
 
 \begin{definition}
For any $f : \mathbb{R}^n \rightarrow \mathbb{C}$ and interval $I  = [ c_I - |I|/2, c_I + |I|/2] \subset \mathbb{R}$, let $f_I (x) := f(|I|^{-1}(x-c_I)) $. 
\end{definition}
 \begin{theorem}\label{P1}
 There exists $\tilde{1} \in C^{\infty}([-1/2,1/2])$ and an almost disjoint collection of intervals $\{I\}=\mathcal{I}$ so that
 \begin{eqnarray*}
\mathcal{V}_2 : f \mapsto \sup_{\tau \in \mathbb{R}} \left( \sum_{I \in \mathcal{I}} \left|f* \mathcal{F}^{-1} \left[ \tilde{1}_{I+\tau} \right] \right|^2 \right)^{1/2}
 \end{eqnarray*}
satisfies no $L^p$ estimates. 
 \end{theorem}

Next, a corollary of L. Grafakos and N. Kalton's work in \cite{MR1853518} is that symbols $m: \mathbb{R}^2 \rightarrow \mathbb{C}$ adapted to the singularity $\Gamma = \{\xi_1 + \xi_2 =0\}$ in the Mikhlin-H\"ormander sense that
 
 \begin{eqnarray*}
\left |\partial^{\vec{\alpha}} m(\vec{\xi})\right| \lesssim_{\vec{\alpha}} \frac{1}{dist(\vec{\xi}, \Gamma)^{|\vec{\alpha}|}}
 \end{eqnarray*}
 for arbitrarily many multi-indices $\vec{\alpha}$ need not be bounded operators on any $L^p(\mathbb{R})$ spaces. In \S{3} we construct an explicit counterexample: 

 \begin{theorem}\label{P2}
 There exists a multiplier $m :\mathbb{R}^2 \rightarrow \mathbb{C}$ adapted to the singularity $\Gamma= \{ \xi_1 + \xi_2 = 0\}$ satisfying 
 
 \begin{eqnarray*}
 \left| \partial^{\vec{\alpha}} m(\vec{\xi})\right| \lesssim_{\vec{\alpha}} \frac{1}{|dist( \vec{\xi}, \Gamma)|^{|\vec{\alpha}|}}
 \end{eqnarray*}
 for all multi-indices $\vec{\alpha}$ such that $T_m: (f_1, f_2) \mapsto \int_{\mathbb{R}^2} m(\xi_1, \xi_2) \hat{f}_1(\xi_1) \hat{f}_2(\xi_2) e^{2 \pi i x (\xi_1 + \xi_2)} d \xi_1 d \xi_2$ (for all $\vec{f} \in \mathcal{S}(\mathbb{R})^2$) satisfies no $L^p$ estimates. 
 \end{theorem}

In \S{4} we prove $L^p(\mathbb{R}) \rightarrow L^p(\mathbb{R})$ estimates when $2 <p <\infty$ for a smooth restricted variant of $V_2$ denoted by $\mathcal{V}_2^{res}$, which is defined via the formula
\begin{eqnarray*}
\mathcal{V}^{res}_2 : f \mapsto \sup_{R \in \mathbb{R}} ~~\sup_{0 \leq \alpha  < R} ~~\left(\sum_{j \in \mathbb{Z}}  \left|f*\mathcal{F}^{-1} \left[ \tilde{1}_{[\alpha + j R, \alpha + (j+1)R]}\right] \right|^2 \right)^{1/2},
\end{eqnarray*}
Note that in moving from $V_2$ to $\mathcal{V}_2^{res}$ we have replaced sharp frequency cutoffs over all increasing sequences $N_1 < ... < N_K$ by mollified cutoffs only over those increasing sequences $N_1 < N_2 < ... < N_K$ for which $N_{j+2} - N_{j+1}= N_{j+1}- N_j$ for all $j \in \{1, ..., K-2\}$. The proof of $\mathcal{V}_2^{res}$ estimates relies on bounding a straightforward time-frequency model.  

In \S{5} and \S{6} we study some bi-sublinear variants of $\mathcal{V}_2$ defined for a set $\Sigma = \{ \sigma \}$ and $\hat{\eta} \in C^\infty([-1/2, 1/2])$ by the formula

\begin{eqnarray*}
\mathcal{M}^{\Sigma}_\eta : (f,g) \mapsto \sup_{\sigma \in \Sigma} \left| \sum_{\tau \in \mathbb{Z}} f*\mathcal{F}^{-1} \left[ \hat{\eta}_{I_\tau^\sigma} \right] g*\mathcal{F}^{-1} \left[ \hat{\eta}_{I^\sigma_0} \right] \right| =: \sup_{\sigma \in \Sigma} \mathcal{M}^\sigma _\eta (f,g) 
\end{eqnarray*}
where $I_\tau ^\sigma := [ \tau \sigma -\sigma/2, \tau \sigma + \sigma/2]$ for each $\tau \in \mathbb{Z}$ and $\sigma \in \mathbb{R}_+$. 
At each scale $\sigma \in \Sigma, \mathcal{M}^\sigma_\eta$  has an adjoint operator at the same scale adapted to $\{\xi_1 + \xi_2 = 0\}$.  Then, by Cauchy-Schwarz, the supremum of the adjoint over all scales satisfies some estimates. However, our main result in these sections shows that no non-trivial mixed estimates for  $\mathcal{M}^\Sigma_\eta$ are possible provided $|\Sigma| = \infty$ and $\eta \not \equiv 0 $ has all non-negative Fourier coefficients. The trivial mixed estimates take the form $W_\infty (\mathbb{R}) \times L^{p_2}(\mathbb{R}) \rightarrow L^{p_2}(\mathbb{R})$ with $1 < p_2 \leq \infty$. That is, we have

\begin{theorem}\label{P3}
Let $\hat{\eta} \in C^\infty ([-1/2,1/2])$ satisfy $\eta \gtrsim 1_{[-1,1]}$.  Let $\Sigma \subset \mathbb{R}$ satisfy $\#\left\{ \Sigma \right\}=\infty$.  Then

\begin{eqnarray*}
\mathcal{M}^\Sigma_\eta : (f, g) \mapsto \sup_{\sigma \in \Sigma} \left| \sum_{\tau \in \mathbb{Z}} f*\mathcal{F}^{-1} \left[ \hat{\eta}_{I_\tau^\sigma} \right] g*\mathcal{F}^{-1} \left[ \hat{\eta}_{I^\sigma_0} \right] \right|
\end{eqnarray*}
maps $W_{p_1}(\mathbb{R}) \times L^{p_2}(\mathbb{R}) \rightarrow L^{p_1p_2/(p_1 + p_2)}(\mathbb{R})$ iff $p_1 = \infty, 1 < p_2 \leq \infty$.
\end{theorem}

Lastly, we investigate in \S{7} a special collection of discretized Hilbert transforms and conclude that any bilinear multiplier with a symbol adapted to the singular line $\{\xi_1 + \xi_2 =0\}$ in the H\"{o}rmander-Mikhlin sense with the added property that along each scale the frequency projections are equally-spaced translated copies of each other maps $L^{p_1}(\mathbb{R}) \times L^{p_2}(\mathbb{R}) \rightarrow L^{p_1 p_2 /(p_1 + p_2)}(\mathbb{R})$ provided $2 \leq p_1, p_2 < \infty$ and $\frac{1}{p_1} + \frac{1}{p_2}$.  More precisely, the following is true:

 \begin{theorem}\label{P4}
Fix $ \hat{\eta} \in C^{\infty} ([-1/2, 1/2])$. Let $m : \mathbb{R}^2 \rightarrow \mathbb{R}$ be given by

\begin{eqnarray*}
m(\xi_1, \xi_2) = \sum_{\vec{P} \in \mathbb{P}} \hat{\eta}_{P_1}(\xi_1)  \hat{\eta}_{P_2}(\xi_2)
\end{eqnarray*}
where

\begin{eqnarray*}
\mathbb{P} := \bigcup_{k \in \mathbb{Z}} \bigcup_{m \in \mathbb{Z}}\left\{  \left[m2^{-k}-2^{-k-1}, m2^{-k} + 2^{-k-1}\right] \times  \left[m2^{-k}-2^{-k-1}+\Gamma 2^{-k} , m2^{-k} + 2^{-k-1}+\Gamma 2^{-k}\right]  \right\}.
\end{eqnarray*}
Then, for every pair $(p_1, p_2)$ such that $2 \leq p_1, p_2 <\infty$ with $\frac{1}{p_1} + \frac{1}{p_2} <1$

\begin{eqnarray*}
|| T_m(f_1, f_2) || _{\frac{p_1 p_2 }{p_1 + p_2}} \lesssim_{\vec{p}} || f_1||_{p_1} || f_2||_{p_2}.
\end{eqnarray*}

\end{theorem}

  \section{ $V_2$ Counterexample}
We first prove that a smooth variant of $V_2$ given by the maximal translation square function  corresponding to a fixed collection of  disjoint intervals $\mathcal{I}$ is unbounded. Recall
 \begin{reptheorem}{P1}
 There exists $\tilde{1} \in C^{\infty}([-1/2,1/2])$ and an almost disjoint collection of intervals $\{I\}=\mathcal{I}$ so that setting $\tilde{1}_I(x) =\check{\tilde{1}}(|I|^{-1}( x-c_I) ) $ for all intervals $I= [c_I - |I|/2, c_I + |I|/2] \subset \mathbb{R}$ ensures the operator
 
 \begin{eqnarray*}
\mathcal{V}_2 : f \mapsto \sup_{\tau \in \mathbb{R}} \left( \sum_{I \in \mathcal{I}} \left|f* \mathcal{F}^{-1} \left[ \tilde{1}_{I+\tau} \right] \right|^2 \right)^{1/2}
 \end{eqnarray*}
satisfies no $L^p$ estimates.

 \end{reptheorem}
 
 \begin{proof}
 Let $\mathcal{I} _k= \left\{ [k+\frac{l}{k}, k + \frac{l+1}{k}] : 0 \leq l <k \right\}$ and define $\mathcal{I} = \bigcup_{k \in \mathbb{N}} \mathcal{I}_k$.  Choose $\tilde{1} \in C^{\infty}([-1/2,1/2])$ to ensure $|\check{\tilde{1}}| \gtrsim 1_{[-2, 2]}$.  Fix $N \in \mathbb{N}$ together with $\phi \in \mathcal{S}(\mathbb{R})$ such that $1_{[-1/4, 1/4]} \leq \hat{\phi} \leq 1_{[-1/2, 1/2]}$ and construct 
 
 \begin{eqnarray*}
 f_N(x) := \sum_{1 \leq n \leq N} \phi(x-n) e^{2 \pi i nx}:= \sum_{1 \leq n \leq N} g_n(x)
 \end{eqnarray*}
in which case $supp~\hat{g}_n \subset [n-1/2, n+1/2]$. It clearly suffices to prove the claim $\mathcal{V}_2 (f_N)(x) \gtrsim \log^{1/2}(N) 1_{[1, N/2]}(x)$, for then $||\mathcal{V}_2(f_N) ||_p \simeq \log(N) N^{1/p}$, whereas $||f_N||_p \simeq N^{1/p}$. 

To this end, let $x \in [j,j+1]$ and set $\tau = j$. Fix $ k \in [10, N-j]$ and consider $I \in \mathcal{I}_{k}+j$ so that $I \subset [k+j-1/4, k+j +1/4]$. Then $g_{k+j} * \check{\tilde{1}}_{I}=\check{\tilde{1}}_I(x-k-j) =k^{-1}\check{\tilde{1}}(k^{-1} (x-k-j))e^{2 \pi i c_I (x-k-j)} $ . There are $O(k)$ many such intervals in $\mathcal{I}_k$ with this property. Therefore, the total contribution from all intervals $I \in \mathcal{I}_k$ at $x \in [j, j+1]$  is at least $\sum_{I \in \mathcal{I}_k: I \subset [k+j-1/4, k+j +1/4]} \left| g_{k+j}*\check{\tilde{1}}_I \right|^2(x) = \sum_{I \in \mathcal{I}_k: I \subset [k+j-1/4, k+j +1/4]} \frac{1}{k^2} \left| \check{\tilde{1}}_I(x-k-j) \right|^2(x)  \gtrsim \frac{1}{k}$. Summing over $10 \leq k \leq N-j$ yields 

\begin{eqnarray*}
\mathcal{V}^2_2(f_N)(x) \gtrsim \sum_{10 \leq k \leq N-j} ~\sum_{I \in \mathcal{I}_k: I \subset [k+j-1/4, k+j +1/4]} \left| g_{k+j}*\check{\tilde{1}}_I \right|^2(x) \gtrsim  \sum_{N-j \geq k \geq 10} \frac{1}{k} \gtrsim \log(N-j). 
\end{eqnarray*}
Choosing $j \in [1, N/2]$ ensures $\log(N-j) \gtrsim \log(N)$. 
\end{proof}

Replacing smooth frequency cutoffs $\left\{ \mathcal{F}^{-1}\left(\tilde{1}_I\right)\right\}_{I \in \mathcal{I}}$ with sharp ones $\{ \mathcal{F}^{-1} ( 1_I)\}_{I \in \mathcal{I}}$ is not a problem. Indeed, the tails of the sharp frequency projections decay even more slowly and one can shrink the size of intervals appearing in the collection $\mathcal{I}$ by a large enough constant factor $A$ to ensure $|\mathcal{F}^{-1} (1_{[-1/A, 1/A]})|\gtrsim_A 1_{[-2,2]}$.
  
  \section{Hilbert Transform Type Multiplier Counterexample}
 \begin{reptheorem}{P2}
 There exists a multiplier $m :\mathbb{R}^2 \rightarrow \mathbb{C}$ adapted to the singularity $\Gamma= \{ \xi_1 + \xi_2 = 0\}$ satisfying 
 
 \begin{eqnarray*}
 \left| \partial^{\vec{\alpha}} m(\vec{\xi})\right| \lesssim_{\vec{\alpha}} \frac{1}{|dist( \vec{\xi}, \Gamma)|^{|\vec{\alpha}|}}
 \end{eqnarray*}
 for all multi-indices $\vec{\alpha}$ such that $T_m: (f_1, f_2) \mapsto \int_{\mathbb{R}^2} m(\xi_1, \xi_2) \hat{f}_1(\xi_1) \hat{f}_2(\xi_2) e^{2 \pi i x (\xi_1 + \xi_2)} d \xi_1 d \xi_2$ (for all $\vec{f} \in \mathcal{S}(\mathbb{R})^2$) satisfies no $L^p$ estimates. 
 \end{reptheorem}
 
 \begin{proof}
Choose $\Phi \in C^\infty([-1/2,1/2])$ real and symmetric such that $\hat{\Phi}(x) >1_{[-1,1]}(x)$. Note that $\hat{\Phi}$ is automatically real and symmetric. 
Let $\Gamma >>1$.  Define a collection of  frequency squares
 
 \begin{eqnarray*}
 \mathbb{Q} &:=&  \bigcup_{k \geq 8 }~ \bigcup_{m \in \mathbb{Z}}~ \bigcup_{-2^{k-8} < \lambda < 2^{k-8}} \vec{Q}_{k,m, \lambda},
\end{eqnarray*}
where for each $ k  \geq 8, m \in \mathbb{Z}, -2^{k-8} < \lambda < 2^{k-8}$

\begin{eqnarray*}
\vec{Q}_{k, m \lambda}  := [m+ \lambda 2^{-k}-2^{-k-1}, m + \lambda 2^{-k}+2^{-k-1} ] \times [ -m - \lambda 2^{-k}-2^{-k-1} + \Gamma 2^{-k}, -m - \lambda 2^{-k} +2^{-k-1} + \Gamma 2^{-k} ] .
\end{eqnarray*}
Next, assign
 \begin{eqnarray*}
 \eta^1_{\vec{Q}_{k,m, \lambda}}(x) &=& 2^{-k} \hat{\Phi}(x2^{-k}) e^{2 \pi i (m + \lambda 2^{-k}) x} \\ 
 \eta^2_{\vec{Q}_{k,m, \lambda}}(x)&=& 2^{-k} \hat{\Phi}(x2^{-k}) e^{-2 \pi i (m + \lambda 2^{-k}- \Gamma 2^{-k}) x} e^{ 2 \pi i \Gamma 2^{-k} m}.
 \end{eqnarray*}
 Let $m(\xi_1 , \xi_2) = \sum_{k\geq 8} \sum_{|\vec{Q}|=2^{-k}} \hat{ \eta}^1_{Q_1} (\xi_1) \hat{\eta}^2_{Q_2}(\xi_2) \in \mathcal{M}_{\{ \xi_1 +\xi_2 =0\}}(\mathbb{R}^2)$. 
 Moreover, letting $\epsilon =1/100$, choose $\phi \in \mathcal{S}(\mathbb{R})$ satisfying 
 
 \begin{eqnarray*}
  1_{[-1/2+\epsilon, 1/2-\epsilon]} \leq \hat{\phi}  \leq1_{[-1/2, 1/2]}.
 \end{eqnarray*}
 In addition, for each $N \in \mathbb{N}$, construct $f_1^N (x) = \sum_{1 \leq n \leq N} \phi(x-n) e^{2 \pi i n x}$ and $f_2^N(x) = \sum_{1 \leq n \leq N} \phi(x-n) e^{-2 \pi i n x}$.  For a given $(c_{Q_1}, c_{Q_2}) = (m + \lambda 2^{-k}, -m - \lambda 2^{-k} + \Gamma 2^{-k})$ for which $[m+\lambda 2^{-k} -2^{-k-1}, m+\lambda 2^{-k}  + 2^{-k-1}] \cap  [n_1-1/2, n_1+1/2] \not = \emptyset  $ and $ [ m+ \lambda 2^{-k} - \Gamma 2^{-k}-2^{-k-1} , m+ \lambda 2^{-k} - \Gamma 2^{-k} + 2^{-k-1}] \cap [n_2-1/2, n_2+ 1/2] \not = \emptyset$, then $m= n_1=n_2$ for all $k \geq C_\Gamma$, in which case
 
 \begin{eqnarray*}
 &  [m+\lambda 2^{-k} -2^{-k-1}, m+\lambda 2^{-k}  + 2^{-k-1}] &\subset [m-1/2+\epsilon, m+1/2-\epsilon] \\ &  [m+\lambda 2^{-k} \Gamma 2^{-k} -2^{-k-1} , m+ \lambda 2^{-k} - \Gamma 2^{-k} + 2^{-k-1}]  &\subset [m-1/2+\epsilon, m+1/2-\epsilon].
 \end{eqnarray*}
Therefore, for each $k \geq C_\Gamma$, we have
 
 \begin{eqnarray*}
 && T^k_m(f_1, f_2)(x)\\&=& \sum_{m \in \mathbb{Z}} \sum_{-2^{k-8} < \lambda < 2^{k-8}}  \sum_{1\leq n_1, n_2 \leq N}  \left (\phi(\cdot -n_1) e^{2 \pi i n_1 \cdot}\right) * \eta^1_{\vec{Q}_{k,m, \lambda}} (x) \left(\phi(\cdot-n_2) e^{-2 \pi i n_2\cdot}\right) * \eta^2_{\vec{Q}_{k, m , \lambda}}(x) \\ &=&  \sum_{1 \leq m \leq N} \sum_{-2^{k-8} < \lambda < 2^{k-8}}   \left (\phi(\cdot -m) e^{2 \pi i m \cdot}\right) * \eta^1_{\vec{Q}_{k,m, \lambda}} (x) \left(\phi(\cdot-m) e^{-2 \pi i m\cdot}\right) * \eta^2_{\vec{Q}_{k, m , \lambda}}(x) \\ &=&  \sum_{1 \leq m \leq N} \sum_{-2^{k-8} < \lambda < 2^{k-8}}  2^{-2k} \hat{\Phi}((x-n)2^k) e^{2 \pi i(m+\lambda 2^{-k})  (x-m)} \hat{\Phi}((x-m)2^k) e^{-2 \pi i (m+\lambda 2^{-k} - \Gamma 2^{-k} ) (x-m)} e^{2 \pi i \Gamma 2^{-k} m} \\ &=& [ 2^{k-7}-1] 2^{-2k} \sum_{1 \leq m \leq N} ( \hat{\Phi}((x-m)2^k)))^2 e^{2 \pi i \Gamma 2^{-k} x}. 
 \end{eqnarray*}
By the assumption $\hat{\Phi}$ is real-valued with $\hat{\Phi}(x) >1_{[-1,1]}(x)$, $|T^k_m(f_1, f_2)(x)| \gtrsim 1_{[1,N]}(x)$ for all $ C_\Gamma \leq k \lesssim \log(N)$.  Lastly, by picking $\Gamma =100$, say, 

\begin{eqnarray*}
supp~ \mathcal{F}(T_m^k(f_1^N, f_2^N)) \subset [99\cdot 2^{-k}, 101\cdot 2^{-k}].
\end{eqnarray*}
Letting $1< p_1, p_2 <\infty$ satisfy $\frac{1}{p_1} + \frac{1}{p_2} <1$, note by Littlewood-Paley equivalence
 
 \begin{eqnarray*}
 \left| \left| T_{m} (f^N_1, f^N_2)  \right| \right|_{\frac{p_1 p_2}{p_1 + p_2}} &=&   \left| \left| \sum_{k \geq 8}T^k_m(f_1^N, f_2^N) \right| \right|_{\frac{p_1 p_2}{p_1 + p_2}}  \\ &\simeq&   \left| \left| \left( \sum_{k \geq 8} \left| T_m^k(f_1^N, f_2^N)\right|^2 \right)^{1/2} \right| \right|_{\frac{p_1 p_2}{p_1 + p_2}} \\ &\gtrsim& \log(N)^{1/2} N^{\frac{1}{p_1}+\frac{1}{p_2}}.
 \end{eqnarray*}
 However, $||f_i||_{p_i} \simeq N^{1/p_i}$ for $i \in \{1,2\}$, so taking N arbitrarily large establishes the claim.

 \end{proof}

 \section{Estimates for $\mathcal{V}_2^{res}$}
We first establish 
 \begin{prop}
The bi-sublinear operator
\begin{eqnarray*}
\mathcal{B} : (f_1, f_2) \mapsto \sup_{ k \in \mathbb{Z}} \left| \sum_{| \vec{P} | = 2^k}  f_1*\eta^1_{P_1} f_2* \eta^2_{-P_1} \right| 
\end{eqnarray*}
maps $L^{p_1}(\mathbb{R}) \times L^{p_2} (\mathbb{R}) \rightarrow L^{p_1 p_2/(p_1 + p_2)} (\mathbb{R})$ provided $2 < p_1, p_2 <\infty$. 
\end{prop}

\begin{proof}
By routine arguments, $\mathcal{B}$ can be linearized, dualized, and discretized to form the model 

\begin{eqnarray*}
\sum_{\vec{P} \in \mathbb{P}}\frac{1}{|I_P|^{1/2}} \langle f_1, \Phi^1_{P_1} \rangle \langle f_2, \Phi^2_{-P_1} \rangle \langle f_3 1_{\{x|N(x) = k\}} , \tilde{1}_{I_P} \rangle
\end{eqnarray*}
where for each dyadic tile $\vec{P} =(I_P, \omega_P) \in \mathbb{P}$,  $\Phi^1_{\vec{P}}$ and $\Phi^2_{\vec{P}}$ have Fourier support in $\omega_P$ and are rapidly decaying away from $I_P$. 
By scaling invariance and interpolation, we only need to establish the corresponding restricted weak type estimates under the assumption $|f_i| \leq 1_{E_i}$ and $|E_3| = 1$. For some $0 \leq \alpha_1, \alpha_2 \leq 1$ to be determined, let 

\begin{eqnarray*}
\Omega_{\alpha_1, \alpha_2} = \left\{ M 1_{E_1} \geq C |E_1|^{\alpha_1} \right\} \bigcup \left\{ M 1_{E_2} \geq C |E_2|^{\alpha_2} \right\}
\end{eqnarray*}
 with $C$ large enough to ensure $|\Omega_{\alpha_1, \alpha_2}| \leq 1/2$ and $\tilde{E}_3 := E_3 \cap \Omega_{\alpha_1, \alpha_2}^c$ is a major subset of $E_3$. For $d \geq 0$, define

\begin{eqnarray*}
\mathbb{P}_d = \left\{ \vec{P} \in \mathbb{P} : 1 + \frac{dist(I_P, \Omega_{\alpha_1, \alpha_2}^c)}{ I_P} \simeq 2^d \right\}. 
\end{eqnarray*}
Let $\mathbb{I}^{1, n_1}$ be the collection of dyadic intervals maximal with respect to the property $I \subset \left\{ M 1_{E_1} \geq 2^{-n_1} \right\}$ and set $\mathbb{P}_d^{1, n_1}= \mathbb{P}_d \cap \left\{ I_P \subset \bigcup_{I \in \mathbb{I}^{1, n_1}} I ~ for~ which~ I_P \not \subset \mathbb{I}^{1, m}~for~any~m < n_1 \right\}$. Similarly define $\mathbb{P}_d^{2, n_2}$ by substituting $E_2$ for $E_1$. Lastly, define $\bar{\mathbb{P}}_d^{n_1, n_2} = \mathbb{P}_d ^{1, n_1} \cap \mathbb{P}_d^{2, n_2}$. 

By construction, 

\begin{eqnarray*}
&&\sum_{\vec{P} \in \bar{\mathbb{P}}^{n_1, n_2}_d} \frac{1}{|I_P|^{1/2}} \langle f_1, \Phi_{P_1} \rangle \langle f_2, \Phi_{-P_1} \rangle \langle f_3 1_{\tilde{E}_3} 1_{\{x|N(x) = k\}} , \tilde{1}_{I_P}   \rangle  \\ &=&  \sum_{k \in \mathbb{Z}}~~ \sum_{\vec{P} \in \mathbb{P}^{n_1, n_2}_d : |I_P|=2^k} \frac{1}{|I_P|} \langle f_1, \Phi_{P_1} \rangle \langle f_2, \Phi_{-P_1} \rangle \langle f_3 1_{\tilde{E}_3}1_{\{x|N(x) = k\}} ,\tilde{1}_{I_P}\rangle \\&\lesssim &\sum_{k \in \mathbb{Z}} \sum_{|I|=2^k} \left( \sum_{\vec{P}\in \mathbb{P}^{n_1, n_2}_d : I_P = I} \frac{ |\langle f_1, \Phi_{P_1} \rangle|^2 }{|I_P|}\right)^{1/2} \left( \sum_{\vec{P}\in\mathbb{P}^{n_1, n_2}_d : I_P = I} \frac{ |\langle f_2, \Phi_{-P_1} \rangle|^2 }{|I_P|}\right)^{1/2} \left| \langle f_31_{\tilde{E}_3} 1_{\{x|N(x) = k\}} , \tilde{1}_{I}\rangle \right| \\ &\lesssim& \sum_{k \in \mathbb{Z}} \sum_{|I| = 2^k } 2^{-n_1/2} 2^{-n_2/2}  \left| \langle f_3 1_{\tilde{E}_3}1_{\{x|N(x) = k\}}, \tilde{1}_{I_P}\rangle \right| \\&\lesssim& 2^{-n_1/2} 2^{-n_2/2} 2^{-C d}. 
\end{eqnarray*}
Therefore, summing over all $n_1, n_2$ such that $2^{-n_1} \lesssim 2^d |E_1|^{\alpha_1}$ and $2^{-n_2} \lesssim 2^d |E_2|^{\alpha_2}$ yields $\left| \Lambda(f_1, f_2, f_31_{\tilde{E}_3})\right| \lesssim \sum_{ d \geq 0} \sum_{ n_1, n_2} |\Lambda_d^{n_1, n_2} (f_1, f_2, f_3 1_{\tilde{E}_3}) | \lesssim  |E_1|^{\alpha_1/2} |E_2|^{\alpha_2/2}$. Setting $\alpha_1 = 2 / p_1, \alpha_2 = 2/p_2$ gives us restricted weak type estimates of the form $(1/p_1, 1/p_2, 1-1/p_1- 1/p_2)$ for all $2 \leq p_1, p_2 \leq \infty$. 

\end{proof}
 \begin{corollary}
The restricted variation 
\begin{eqnarray*}
\mathcal{V}^{res}_2 : f \mapsto \sup_{R \in \mathbb{R}_+} ~~\sup_{0 \leq \alpha  < R} ~~\left(\sum_{j \in \mathbb{Z}}  \left|f*\mathcal{F}^{-1} \left[ \tilde{1}_{[\alpha + j R, \alpha + (j+1)R]}\right] \right|^2 \right)^{1/2} 
\end{eqnarray*}

maps $L^p(\mathbb{R}) \rightarrow L^p(\mathbb{R})$ for all $2 < p <\infty$. 
 \end{corollary}
 \begin{proof}
 Expand bumps functions on shifted non-dyadic intervals as Fourier series on dyadic intervals. 
 \end{proof}

For completeness, we include the following negative result:
\begin{prop}
The operator $\mathcal{V}_2^{res}$ does not map $L^2 (\mathbb{R}) \rightarrow L^2(\mathbb{R})$. 
\end{prop}

\begin{proof}
Fix $\phi \in \mathcal{S}(\mathbb{R})$ such that $1_{[-1, 1]} \leq \hat{\phi} \leq 1_{[-2,2]}$. Fix $|x| \simeq 2^{l}$ for some $l >0$.  Then $\sum_{|Q| = 2^{l}} |\phi*\eta_Q(x)|^2 \simeq 2^{-l} $. Hence, 

\begin{eqnarray*}
\sup_{l \in \mathbb{Z}} \left( \sum_{|Q| = 2^{-l}} |\phi * \eta_{Q}(x)|^2 \right)^{1/2} \gtrsim \frac{ 1}{1+|x|^{1/2}}.
\end{eqnarray*}
Clearly, $\phi \in L^2(\mathbb{R})$ and $\frac{1}{1+|x|^{1/2}} \not \in L^2(\mathbb{R})$. 
\end{proof}

\section{Counterexamples for a Maximal Adjoint}

\begin{definition}
For $\tau \in \mathbb{Z}$ and $\sigma \in \mathbb{R}_+$ let $I_\tau ^{\sigma} := [ \tau \sigma  - \sigma/2 , \tau \sigma + \sigma/2].$
\end{definition}

 \begin{prop}
Let $\hat{\eta} \in C^\infty ([-1/2,1/2])$ satisfy $\eta \gtrsim 1_{[-1,1]}$. Then 
\begin{eqnarray*}
\mathcal{M}_\eta: (f, g) \mapsto \sup_{k \in \mathbb{R}_+}\left|  \sum_{\tau\in \mathbb{Z} } f*\mathcal{F}^{-1} \left[ \hat{\eta}_{I_\tau^{k}} \right]  g*\mathcal{F}^{-1} \left[ \eta_{I_0^{k}}\right] \right|
\end{eqnarray*}
maps $L^{p_1}(\mathbb{R}) \times L^{p_2}(\mathbb{R}) \rightarrow L^{p_1 p_2 / (p_1 + p_2)}(\mathbb{R})$ iff $p_1=\infty$ and $1 < p_2 \leq \infty$.  

\end{prop}
\begin{proof}
 Restrict $k \geq 1$. Let $\hat{\eta} (x) = \sum_{ n \in \mathbb{Z}} c_n e^{2 \pi i n x}$ so that $f \mapsto \sum_{ \tau \in \mathbb{Z}} f* \mathcal{F}^{-1} \left[ \eta_{I^{k}_\tau} \right]$ has multiplier given by $m_k(\xi) = \sum_{n \in \mathbb{Z}} c_n e^{2 \pi i nk^{-1} \xi}$, and the operator may be represented as $f \mapsto \sum_{n \in \mathbb{Z}} c_n f(x-k^{-1} n)$.  Fix $k_0 \in \mathbb{N}$. Note $c_1 =\eta(-1) =  \int_{\mathbb{T}} \hat{\eta}(\xi) e^{-2 \pi i \xi} d\xi >0$.  Pick $f$ satisfying $1_{[-k_0^{-1}, k_0^{-1}]} \leq f \leq 1_{[-2k_0^{-1},2 k_0^{-1}]}$ and $g=1_{[-1,1]}$. Then $g * \eta_{I_0^k} \gtrsim 1_{[-1,1]} $ for all $k \geq 1$. Thus,  $\sup_{k \in \mathbb{R}} \left| \sum_{\tau \in \mathbb{Z}} f*\eta_{I_\tau^k}(x)\right| =\sup_{k \in \mathbb{Z}} \left| \sum_{n \in \mathbb{Z}} c_n f(x-k^{-1} n)\right|$. As $c_n =\eta(-n) \geq 0$ for all $n$, it suffices to observe
 
 \begin{eqnarray*}
 \mathcal{M}(f, g)(x) &\gtrsim& \sup_{k \geq 1} \left| \sum_{n \in \mathbb{Z}} c_n f(x-k^{-1} n) 1_{[-1,1]} (x)\right| \\ &\geq&  \sup_{k \geq 1} \left| c_1  f(x-k^{-1} ) \right| 1_{[-1,1]} (x) \\ &\gtrsim& 1_{[-1,1]}(x). 
 \end{eqnarray*}
 Therefore, $\left| \left|  \mathcal{M}(f_{k_0}, g) \right| \right|_{\frac{p_1p_2}{p_1 + p_2}} \gtrsim 1$ while $||f_k||_{p_0} \sim k_0^{-1/p_1}$ and $||g||_{p_2} \simeq 1$. If $p_1 = \infty$, then estimates are trivially satisfied by virtue of $\mathcal{M}(f, g)(x) \lesssim ||f||_\infty Mg(x). $

\end{proof}

\begin{lemma}
Fix $k_0 \in \mathbb{N}$. Let $S \subset [1, 2^{k_0}] \cap \mathbb{N}.$ Then there exists $n \in [-2^{-k_0}, 2^{k_0}] \cap \mathbb{Z}$ such that

\begin{eqnarray*}
\left| \left\{ 2^k + n \right\}_{0 \leq k < k_0} \cap S \right|  \geq  \frac{ k_0 |S|}{2 2^{k_0}+1}.
\end{eqnarray*}

\end{lemma}

\begin{proof}
It suffices to note

\begin{eqnarray*}
\sum_{|n| \leq 2^{k_0}} \sum_{1 \leq m \leq 2^{k_0}} 1_S(m) 1_{\{ 2^k +n\}_{0 \leq k < k_0}}(m) =\sum_{1 \leq x \leq 2^{k_0}} 1_S(m)\sum_{|n| \leq 2^{k_0}}  1_{\{ 2^k +n\}_{0 \leq k < k_0}}(m) = k_0 |S|. 
\end{eqnarray*}

\end{proof}

\begin{prop}
Fix $k_0 \in \mathbb{N}$. Then there exists a set $\mathcal{N}_{k_0} \in [-2^{k_0}, 2^{k_0}] \cap \mathbb{Z}$ satisfying $|\mathcal{N}_{k_0}|\sim 2^{k_0} /k_0$ and

\begin{eqnarray*}
\left| \bigcup_{ n \in \mathcal{N}_{k_0}} \left\{ 2^k +n \right\}_{ 0 \leq k  < k_0} \right| \geq 2^{k_0}/2. 
\end{eqnarray*}

\end{prop}

\begin{proof}
Initialize $S_0 = [1, 2^{k_0}]$. Then select $S_1 = S_0 \cap \left[ \{2^k\}_{0 \leq k<k_0}\right]^c$. Apply the proceeding lemma to $S= S_1$. This yields an $n_1$ such that 

\begin{eqnarray*}
\left| \left\{ 2^k\right\} \bigcup  \left\{ 2^k + n_1 \right\} \right| \gtrsim 2 k_0. 
\end{eqnarray*}
Inductively, we obtain $S_\mu: |S_\mu| \leq |S_{\mu-1}| \left[ 1-\frac{ k_0}{2 \cdot 2^{k_0}} \right]$ and points $n_\mu$ for which 

\begin{eqnarray*}
\left| [1, 2^{k_0} ] \cap \left\{ \bigcup_{1  \leq \gamma \leq \mu} \left\{ 2^k + n_\mu \right\}_{1 \leq k < k_0} \right\} ^c \right| = |S_\mu|. 
\end{eqnarray*}
 Therefore, $|S_\mu| \leq \left[ 1-\frac{ k_0}{2 \cdot 2^{k_0}} \right]^\mu 2^{k_0} \leq 2^{k_0} /2$ so long as 
 
 \begin{eqnarray*}
\mu \gtrsim \frac{1}{|\log(1-\frac{k_0} {2 2^{k_0}})|} \sim \frac{2^{k_0}}{k_0}.
 \end{eqnarray*}

\end{proof}
\begin{prop}
Let $\hat{\eta} \in C^\infty ([-1/2,1/2])$ satisfy $\eta \gtrsim 1_{[-1,1]}$. Then

\begin{eqnarray*}
\mathcal{M}_\eta : (f, g) \mapsto \sup_{k \in \mathbb{Z}} \left| \sum_{ \tau \in \mathbb{Z}} f*\mathcal{F}^{-1} \left[\hat{ \eta}_{I^{2^k}_\tau} \right]  g *\mathcal{F}^{-1} \left[  \hat{\eta}_{I^{2^k}_0} \right] \right|
\end{eqnarray*}
maps $L^{p_1}(\mathbb{R}) \times L^{p_2}(\mathbb{R}) \rightarrow L^{p_1p_2/(p_1 + p_2)}(\mathbb{R})$ iff $p_1 = \infty, 1 < p_2 \leq \infty$. \end{prop}
\begin{proof}
Clearly,  $M_\eta (f, g) (x) \lesssim ||f||_\infty Mg(x)$, so all estimates of the form $L^{\infty} \times L^{p_2} \rightarrow L^{p_2}$ are available. For the other direction, observe $\hat{\eta}(x) = \sum_{n \in \mathbb{Z}} c_n e^{2 \pi i n x}$ satisfies $c_n \geq 0$ for all $n \in \mathbb{Z}$ and $c_1>0$.
Let $f_{k_0} = \sum_{ n \in \mathcal{N}_{k_0}} f(2^{k_0} (x-n))$ with $f=1_{[-2^{-k_0}, 2^{-k_0}]} $  and $g=1_{[-1,1]}$. Then $||f_{k_0}||_{p_1} \sim k_0^{-1/p_1}, ||g||_{p_2} \sim 1$, while $\left| \left|\mathcal{M}_\eta (f_{k_0}, g) \right| \right|_{p_1 p_2/ (p_1 + p_2)} \sim 1$. As $p_1 \not = \infty$, taking $k_0 \rightarrow \infty$ yields the claim. 
\end{proof}

\begin{lemma}
Fix $k_0 \in \mathbb{N}$. Let $\mathcal{K} \subset [1, 2^{k_0-1}]$ satisfy $|\mathcal{K}| = k_0$ and assume $\min_{x \not =y \in \mathcal{K}}  |x-y| \geq 1$. Then there exists a set $\mathcal{N}_{k_0} \subset [-2^{k_0}, 2^{k_0}]$ satisfying $|\mathcal{N}_{k_0}| \sim 2^{k_0}/k_0$ and 
\begin{eqnarray*}
\left| \bigcup_{k \in \mathcal{K}}\bigcup_{\theta \in \mathcal{N} }[k+ \theta -1/2, k+ \theta +1/2] \cap [1, 2^{k_0}] \right|  \geq  2^{k_0}/2.
\end{eqnarray*}

\end{lemma}

\begin{proof}
For any set $X \subset \mathbb{R}$, let $N_1[X] := \left\{ x \in \mathbb{R} : dist(x, X) < 1/2 \right\}$. It suffices to note that for any subset $S \subset [1, 2^{k_0}]$ 

\begin{eqnarray*}
\int_{[-2^{k_0}, 2^{k_0}]} \int_{\mathbb{R}} 1_S(x) 1_{N_{1/2}}\left[\mathcal{K}+n \right](x) dx dn=\int_{\mathbb{R}} 1_S(x)\int_{[-2^{k_0}, 2^{k_0}]}  1_{N_{1/2}}\left[ \mathcal{K}+n \right](x)dn dx = k_0 |S|. 
\end{eqnarray*}
Therefore, there exists  $\theta_1:-2^{k_0} \leq \theta_1 \leq  2^{k_0}$ such that 

\begin{eqnarray*}
\left|  \bigcup_{k \in \mathcal{K}} [ k +\theta_1-1/2, k + \theta_1 + 1/2]  \cap S \right| \geq \frac{ k_0 |S|} {2^{k_0+1}}.
\end{eqnarray*}
Iterate this process exactly as before to obtain a set $\mathcal{N}$ of size $|\mathcal{K}| /k_0$ for which 

\begin{eqnarray*}
\left| \bigcup_{\theta \in \mathcal{N}_{k_0}(\mathcal{K})} \bigcup_{k \in \mathcal{K}} [ k +\theta-1/2, k + \theta + 1/2] \cap [1, 2^{k_0} ]\right| \geq 2^{k_0} /2. 
\end{eqnarray*}

\end{proof}
\begin{prop}
Let $\hat{\eta} \in C^\infty ([-1/2,1/2])$ satisfy $\eta \gtrsim 1_{[-1,1]}$.  Let $\Sigma \subset \mathbb{R}$ satisfy  $\#\left\{ \Sigma \right\}=\infty$.  Then

\begin{eqnarray*}
\mathcal{M}^\Sigma_\eta : (f, g) \mapsto \sup_{\sigma \in \Sigma} \left| \sum_{ \tau \in \mathbb{Z}} f*\mathcal{F}^{-1} \left[ \hat{\eta}_{I^{\sigma}_\tau} \right] g *\mathcal{F}^{-1} \left[  \hat{\eta}_{I^{\sigma}_0} \right] \right|
\end{eqnarray*}
maps $L^{p_1}(\mathbb{R}) \times L^{p_2}(\mathbb{R}) \rightarrow L^{p_1p_2/(p_1 + p_2)}(\mathbb{R})$ iff $p_1 = \infty, 1 < p_2 \leq \infty$. \end{prop}
\begin{proof}
Fix $k_0 \in \mathbb{N}$. By scaling invariance, we may assume $\sigma \geq 2$ for at least $k_0$ many $\sigma \in \Sigma$. Denote this collection by $\Sigma_{k_0}$. Let $\Sigma^{-1}_{k_0}$ be the collection of reciprocals in $(0,1/2]$ and let $\delta_{k_0}$ be smaller than the minimum distance between two elements in $\Sigma^{-1}_{k_0}$ such that $ \lim_{k_0 \rightarrow \infty} k_0 \delta_{k_0}=0$. Then we may set $f(x)= 1_{[-\mathfrak{d}_k, \mathfrak{d}_k]}(x)$, $f_{k_0}(x) = \sum_{ n \in  \mathfrak{d}_k \mathcal{N}_{k_0}(\mathfrak{d}_k^{-1} \Sigma_{k_0}^{-1}) } f(\mathfrak{d}_{k_0}^{-1}(x-n))$, and $g=1_{[-1,1]}$. It is easy to observe $\mathcal{M}(f_{k_0}, g) \gtrsim 1_{\mathcal{S}_{k_0}}$ for some $\mathcal{S}_{k_0} \subset [0,1]$ with $|\mathcal{S}_{k_0}| \simeq 1$ while $||f_{k_0}||_{p_1} \sim k_0^{-1/p_1}$ and $||g||_{p_2} \sim 1$.
\end{proof}

\section{Mixed Counterexamples for a Maximal Adjoint}
\begin{definition}
For any $2 \leq p_1 \leq \infty$, let $W_{p_1}(\mathbb{R}) := \left\{ f \in L^{p_1}(\mathbb{R}): \hat{f} \in L^{p_1^\prime}(\mathbb{R}) \right\}$ with $||f ||_{W_{p_2}(\mathbb{R})}:= || \hat{f}||_{L^{p_1^\prime}(\mathbb{R})}$.  
\end{definition}
The Hausdorff-Young inequality says $|| f||_{L^{p_1}(\mathbb{R})} \leq  ||f ||_{W_{p_1}(\mathbb{R})} $ whenever $2 \leq p_2 \leq \infty$. Therefore, despite the fact that no $L^p$ estimates are available for $\mathcal{M}_\eta$, it is natural to ask whether any mixed estimates of the form $W_{p_1}(\mathbb{R}) \times L^{p_2}(\mathbb{R}) \rightarrow L^{p_1 p_2 / (p_1 + p_2)}(\mathbb{R})$ hold for $2 < p_1 \leq \infty$ and $1 \leq p_1 \leq \infty$. This section shows that there are no non-trivial positive answers to the above question. 

\begin{prop}
Let $\hat{\eta} \in C^\infty ([-1/2,1/2])$ satisfy $\eta \gtrsim 1_{[-1,1]}$. Then the maximal dyadic operator

\begin{eqnarray*}
\mathcal{M}_\eta : (f, g) \mapsto \sup_{k \in \mathbb{Z}} \left| \sum_{ \tau \in \mathbb{Z}} f*\mathcal{F}^{-1} \left[ \hat{\eta}_{I^{2^k}_\tau} \right] g *\mathcal{F}^{-1} \left[ \hat{ \eta}_{I^{2^k}_0}\right] \right|
\end{eqnarray*}
maps $W_{p_1}(\mathbb{R}) \times L^{p_2}(\mathbb{R}) \rightarrow L^{p_1p_2/(p_1 + p_2)}(\mathbb{R})$ iff $p_1 = \infty, 1 < p_2 \leq \infty$. \end{prop}
\begin{proof}
It suffices to prove unboundedness for $(f,g) \mapsto \sup_{k \in \mathbb{Z}} \left| f(\cdot-2^{-k}) g*\eta_{I_0^{2^k}} \right|$. To this end, we exploit the structure of the set $\{2^k\}_{k \in \mathbb{Z}}$ using a few elementary number theoretic facts: for every $m \geq 1$, the orbit of $\{2^k~\mod 5^m \}_{0 \leq k < 4 \cdot 5^{m-1}}$ consists of $4 \cdot 5^{m-1}$ distinct points.  This follows from the fact that $\phi (5^m) = 4 \cdot 5^{m-1}$, $2$ is a primitive root of $(\mathbb{Z}/ (5 \mathbb{Z}))^\times$, and $2^{5-1} \not \equiv 1 \mod 5^2$ ensure $2$ is a primitive root $\mod 5^m$ for all $m \geq 1$. Next, fix $k_0 \in \mathbb{N}$. Choose $m_0$ so that $ 5^{m_0} \sim k_0$. Then observe

\begin{eqnarray*}
\left| \bigcup_{0 \leq \tau \leq 2^{k_0} /k_0} \bigcup_{0 \leq k \leq k_0} \left[2^{k} + \tau 5^{m_0} -1/2, 2^k + \tau 5^{m_0}  +1/2 \right] \right| \gtrsim 2^{k_0} .
\end{eqnarray*}
This observation enables us to choose $f_{k_0} (x) = \sum_{|\tau| \leq 2^{k_0}/k_0} f(2^{k_0}(x+ \tau 5^{m_0}))$ and with $g=1_{[-1,1]}$ in which case $\left| \left|\mathcal{M}_\eta (f_{k_0}, g) \right| \right|_{p_1p_2/(p_1 + p_2)} \gtrsim 1$ and $||g||_{p_2} \sim 1$. It remains to show $||\hat{f}_{k_0}||_{p_1^\prime} \sim k_0^{-1/p_1}$. By rescaling, it clearly suffices to the bound 

\begin{eqnarray*}
\left| \left| \sum_{\lambda \in \Lambda} e^{2 \pi i \lambda \cdot} \right| \right|_{L^{p^\prime}(\mathbb{T})} \sim |\Lambda|^{1/p}
\end{eqnarray*}
whenever $1 \leq p^\prime \leq 2$ and $\Lambda$ is an arithmetic progression of length $|\Lambda|$. Suppose $\Lambda = \left\{ \sigma + \gamma k \right\}_{k_0 \leq k \leq k_1}$. Then 

\begin{eqnarray*}
\sum_{\lambda \in  \Lambda} e^{2 \pi i \lambda x} = e^{2 \pi i \sigma x} \frac{ e^{2 \pi i \gamma k_0 x} \left[ 1- e^{2 \pi i \gamma (k_1 - k_0  + 2)(x)}\right]}{1-e^{2 \pi i \gamma x} }. 
\end{eqnarray*}
It is very simple to observe $\left| \left| \sum_{\lambda \in  \Lambda} e^{2 \pi i \lambda \cdot} \right| \right|_{L^{p^\prime}(\mathbb{T})} =\left| \left| \frac{ \sin ( \pi \gamma  \left[  k_1-k_0+2 \right] x)}{\sin ( \pi \gamma x)} \right| \right|_{L^{p^\prime}(\mathbb{T})}$. Changing variables yields

\begin{eqnarray*}
\gamma^{-1/p^\prime} \left| \left| \frac{ \sin ( \pi  \left[  k_1-k_0+2 \right] x)}{\sin ( \pi x)} \right| \right|_{L^{p^\prime}(\gamma \mathbb{T})} \simeq  \left| \left| \frac{ \sin ( \pi  \left[  k_1-k_0+2 \right] x)}{\sin ( \pi x)} \right| \right|_{L^{p^\prime}(\mathbb{T})} \sim |k_1- k_0|^{1/p} \sim |\Lambda|^{1/p}. 
\end{eqnarray*}
 Indeed, the routine computations are as follows: the integrand is $\lesssim |k_1 - k_0 +2| $ on a set of size $\frac{1}{|k_1-k_0| }$. On $|x| \gtrsim \frac{1}{|k_1-k_2 + 2|}$, the integrand is at most $\frac{1}{|x|}$. Integration yields $ \left[  \frac{1}{|k_1-k_2 + 2|} \right]^{(-p^\prime +1)/ p^\prime} =  \left[ |k_1-k_2 + 2| \right]^{1/p}  $. 

\end{proof}



Note that we ruled out non-trivial mixed estimates for $\mathcal{M}_\eta$ by relying on two facts: (1) for each $k_0 \in \mathbb{N}$, $\{2^k : 1 \leq k \leq k_0 \}$ consists of approximately $k_0$ distinct integers modulo some integer of the same magnitude as $k_0$; (2) the characteristic function of some $\delta-$ neighborhood of an arithmetic progression is a quasi-extremizer for the Hausdorff-Young inequality. It is clear, however, that ruling out non-trivial mixed estimates for $\mathcal{M}^\Sigma_\eta$ in the case of a generic infinite set $\Sigma$ cannot rely on fact $(1)$ and so requires other insights.  Of course, fact $(2)$ has nothing to do with  $\Sigma$ and therefore remains at our disposal. 

The following lemma say that for any infinite set $\mathcal{S}\subset \mathbb{R}$ with a suitable growth condition it is the case that for  any $k_0 \in \mathbb{N}$ there is always some (potentially non-integer) length $\mathcal{L}_{k_0}(\mathcal{S})  \sim k_0$ so that the first $k_0$ elements of $\mathcal{S}$ are essentially disjoint modulo $\mathcal{L}_{k_0}(\mathcal{S})$. More precisely, we have

\begin{lemma}\label{HL}
Let $\mathcal{S}=\left\{ \alpha_j \right\} \subset \mathbb{R} ^+$ with $\alpha_1=1$ and rapidly increasing in the sense that $\alpha_{j+1}  \geq 2^j  \alpha_j $ for all $j \in \mathbb{N}$. Then there exists an absolute constant $C>0$ such that for every $k \in \mathbb{N}$ there exists $\theta_k \sim 1/k$ so that $|| \alpha_j \theta_k||_\mathbb{T} \sim j /k$ for all $\lceil C \log(k) \rceil\leq j \leq k$. 
\end{lemma}
\begin{proof}
Note $\alpha _j \geq 2^{j-1} $.  Choose an absolute constant $C$ large enough so that there exists $\theta =:\theta_{\lceil C \log(k) \rceil} \geq 1/k$ satisfying $0 \leq \theta_{\lceil C \log(k) \rceil}- 1/k \leq  1/\alpha_{\lceil C \log(k)\rceil} \leq 1/(8k)$ and $|| \alpha_{\lceil C\log(k) \lceil} \theta_{\lceil C \log(k) \rceil} ||_{\mathbb{T}}= 1/k$. Next, choose $\theta_2$ so that $|| \alpha_{\lceil C \log(k) \rceil+1} \theta_{\lceil C \log(k) \rceil+1}||_{\mathbb{T}} = 2/k$ and $0 \leq \theta_2 - \theta_1 \leq 1/ \alpha_{\lceil C  \log(k) \rceil+1} $. Observe that 

\begin{eqnarray*}
\left| || \alpha_{\lceil C \log(k) \rceil}\theta_{\lceil C\log(k) \rceil+1}  ||_\mathbb{T} - || \alpha_{\lceil C \log(k) \rceil} \theta_{\lceil C \log(k) \rceil}||_\mathbb{T} \right| &\leq& || \alpha_{\lceil C \log(k) \rceil} (\theta _{\lceil C \log(k) \rceil  +1} - \theta_{\lceil C \log(k) \rceil }) ||_\mathbb{T} \\&\leq& 2^{-C \log(k)}. 
\end{eqnarray*}
Next, choose $\theta_{\rceil C\log(k)  \rceil+2}$ satisfying $||\alpha_{ \lceil C  \log(k) \rceil +2} \theta_{\lceil C \log(k) \rceil +2}||_\mathbb{T} =3/k$ and such that $0 \leq \theta_3 - \theta_2 \leq  1/ \alpha_{\lceil C  \log(k) \rceil+2}$. This ensures 

\begin{eqnarray*}
&& \left| || \alpha_{\lceil C\log(k) \rceil } \theta _{\lceil C \log(k) \rceil+2} || _\mathbb{T} - || \alpha_{\lceil C \log(k) \rceil }  \theta _{\lceil C \log(k) \rceil} || _\mathbb{T} \right| \\&\leq&  || \alpha_{\lceil C  \log(k)\rceil} (\theta _{\lceil C \log(k) \rceil  +2}-\theta_{\lceil C \log(k)\rceil }) || _\mathbb{T}\\ & \leq & || \alpha_{\lceil C  \log(k) \rceil} (\theta _{\lceil C\log(k) \rceil +2}-\theta_{\lceil C \log(k) \rceil +1}) || _\mathbb{T}+ || \alpha_{\lceil C  \log(k) \rceil } (\theta _{\lceil C \log(k)\rceil  +1}-\theta_{\lceil C\log(k)\rceil }) || _\mathbb{T} \\ &<& \frac{ \alpha_{\lceil C  \log(k) \rceil }}{ \alpha_{\lceil C \log(k) \rceil  +2}} + \frac{\alpha_{\lceil C  \log(k)\rceil }}{\alpha_{\lceil C  \log(k) \rceil +1}} \\ &\leq&   2^{-2C \log (k)  -1}  + 2^{-C \log (k)}. 
\end{eqnarray*}
Iterating this construction yields $\theta_k \sim 1/k$ so that $|| \alpha_j \theta_k||_\mathbb{T} \sim j/k$ for all $\lceil C \log(k)  \rceil \leq j \leq k$. 
\end{proof}
\begin{reptheorem}{P3}
Let $\hat{\eta} \in C^\infty ([-1/2,1/2])$ satisfy $\eta \gtrsim 1_{[-1,1]}$.  Let $\Sigma \subset \mathbb{R}$ satisfy $\#\left\{ \Sigma \right\}=\infty$.  Then

\begin{eqnarray*}
\mathcal{M}^\Sigma_\eta : (f, g) \mapsto \sup_{\sigma \in \Sigma} \left| \sum_{ \tau \in \mathbb{Z}} f*\mathcal{F}^{-1} \left[ \hat{\eta}_{I^{\sigma}_\tau} \right] g * \mathcal{F}^{-1} \left[ \hat{\eta}_{I^{\sigma}_0}\right] \right|
\end{eqnarray*}
maps $W_{p_1}(\mathbb{R}) \times L^{p_2}(\mathbb{R}) \rightarrow L^{p_1p_2/(p_1 + p_2)}(\mathbb{R})$ iff $p_1 = \infty, 1 < p_2 \leq \infty$.
 \end{reptheorem}

\begin{proof}
Fix $k_0 \in \mathbb{N}$. By scaling and translation invariance, we may assume $\sigma \geq 2$ for $k_0$ many $\sigma \in \Sigma$ which satisfy $2 = \sigma_1 \leq 2^{-1} \cdot \sigma_2 \leq 2^{-1-2} \sigma_3 \leq ... \leq 2^{-1-2- ... - k_0} \sigma_{k_0}$. That is,  $\sigma_j \leq 2^{-j} \sigma_{j+1}$ for all $j \in \{1, ..., k_0-1\}$.  Denote this collection by $\Sigma_{k_0}$. Let $\Sigma^{-1}_{k_0}$ be the collection of reciprocals in $(0,1/2]$ so that $\sigma_{k_0} \Sigma_{k_0}^{-1}$ satisfies the conditions of Lemma \ref{HL}. Therefore, invoke Lemma \ref{HL} to find $\theta_{k_0}$. Set $g(x) = \phi (x):= \langle x \rangle ^{-10} $ together with $f_{k_0}(x) = \sum_{| \tau| \leq \lceil \sigma_{k_0} \theta_{k_0} \rceil } \phi(\sigma_{k_0}(x- \tau /(\sigma_{k_0}\theta_{k_0})))$. Let $g= \phi$.   For each $j : C\log(k_0) \leq j \leq k_0$,  there exists $\tau:0 \leq  \tau \leq \lceil \sigma_{k_0} \theta_{k_0} \rceil $ so that $0 \leq \alpha_j \theta_{k_0}  - \tau< 1$. Then $\sigma_{k_0} \sigma_{k_0 - j +1}^{-1}:=\alpha_j  = \frac{j}{k_0 \theta_{k_0}}    +\frac{\tau}{ \theta_{k_0}} + O(1)$.  Let  $S_{k_0}: = \bigcup_{0 \leq \tau \leq \lceil \sigma_{k_0} \theta_{k_0} \rceil}\bigcup_{ C\log(k_0) \leq j \leq k_0}  \left\{\frac{  j }{\sigma_{k_0} k_0 \theta_{k_0}} + \frac{ \tau }{\sigma_{k_0}\theta_{k_0}} \right\}$ and $\mathcal{S}_{k_0} := \left\{ x: dist(x ,S_{k_0} ) \leq 1/ \sigma_{k_0} \right\}$.  Then $\mathcal{M}_\eta (f_{k_0}, g) \gtrsim 1_{\mathcal{S}_{k_0}}$ and $|\mathcal{S}_{k_0}| \simeq 1$, yet $||\hat{f}_k||_{p_1^\prime} \simeq k^{-1/p_1}$ and $||g||\simeq 1$. 
\end{proof}

 \section{Estimates for Periodically Discretized Hilbert Transforms in the Plane}
As we have seen, symbols $m: \mathbb{R}^2 \rightarrow \mathbb{C}$ adapted to the singularity $\Gamma = \{\xi_1 + \xi_2 =0\}$ in the Mikhlin-H\"ormander sense that
 
 \begin{eqnarray*}
\left |\partial^{\vec{\alpha}} m(\vec{\xi})\right| \lesssim_{\vec{\alpha}} \frac{1}{dist(\vec{\xi}, \Gamma)^{|\vec{\alpha}|}}
 \end{eqnarray*}
 for arbitrarily many multi-indices $\vec{\alpha}$ need not be bounded operators on any $L^p(\mathbb{R})$ spaces. Of course, there are non-trivial multipliers obeying the above inequality such as $H_2:(f,g) \mapsto H(f \cdot g)$, which clearly satisfy all Banach $L^p$ estimates. Our next result says that any bilinear multiplier consisting of frequency localized pieces arranged in a Whitney decomposition with respect to the singular line $\{\xi_1 + \xi_2 =0\}$ with the additional property that at each scale its frequency projections are equally-spaced, translated copies of each other maps $L^{p_1}(\mathbb{R}) \times L^{p_2}(\mathbb{R}) \rightarrow L^{p_1 p_2 / (p_1 + p_2)}(\mathbb{R})$ for all $2 \leq  p_1, p_2 <\infty$ satisfying $\frac{1}{p_1}+\frac{1}{p_2} <1$. 
  \begin{reptheorem}{P4}
Fix $ \hat{\eta} \in C^{\infty} ([-1/2, 1/2])$. Let $m : \mathbb{R}^2 \rightarrow \mathbb{R}$ be given by

\begin{eqnarray*}
m(\xi_1, \xi_2) = \sum_{\vec{P} \in \mathbb{P}} \hat{\eta}_{P_1}(\xi_1)  \hat{\eta}_{P_2}(\xi_2)
\end{eqnarray*}
where

\begin{eqnarray*}
\mathbb{P} := \bigcup_{k \in \mathbb{Z}} \bigcup_{m \in \mathbb{Z}}\left\{  \left[m2^{-k}-2^{-k-1}, m2^{-k} + 2^{-k-1}\right] \times  \left[m2^{-k}-2^{-k-1}+\Gamma 2^{-k} , m2^{-k} + 2^{-k-1}+\Gamma 2^{-k}\right]  \right\}.
\end{eqnarray*}
Then, for every pair $(p_1, p_2)$ such that $2 \leq p_1, p_2 <\infty$ with $\frac{1}{p_1} + \frac{1}{p_2} <1$

\begin{eqnarray*}
|| T_m(f_1, f_2) || _{\frac{p_1 p_2 }{p_1 + p_2}} \lesssim_{\vec{p}} || f_1||_{p_1} || f_2||_{p_2}.
\end{eqnarray*}
\end{reptheorem}

\begin{proof}
We may assume that $|\vec{P}| \leq 1$ for all $\vec{P} \in \mathbb{P}$ using scaling invariance and a standard limiting argument. In addition, we may assume $f_1, f_2 \in \mathcal{S}(\mathbb{R})$ by density.   Next, by expanding each function in Fourier series on each translated smooth interval $I_n= [n-1/2, n+1/2]$, we face for $\bar{1}_{I_n}(x):= \bar{1}(x-n)$ for some $\bar{1} \in \mathcal{S}(\mathbb{R})$ satisfying $1_{[-1/4,  1/4]} \leq \bar{1} \leq 1_{[3/4, 3/4]}$

\begin{eqnarray*}
f_1(x) &=& \sum_{n \in \mathbb{Z}} f_1(x) \bar{1}_{I_n}(x) =  \sum_{n \in \mathbb{Z}} \sum_{ \mu \in \mathbb{Z}} c_{1,n}^\mu e^{2 \pi i \frac{2}{3} \mu (x-n)}\bar{1}_{I_{n}} (x)  \\ 
f_2(x) &=&  \sum_{n \in \mathbb{Z}} f_2(x) \bar{1}_{I_n}(x) =  \sum_{n \in \mathbb{Z}} \sum_{ \mu \in \mathbb{Z}} c_{2,n}^\mu e^{2 \pi i \frac{2}{3} \mu (x-n)}\bar{1}_{I_{n}} (x) .
\end{eqnarray*}
By standard L-P equivalence, we have that

\begin{eqnarray*}
\left| \left|   \sum_{k \geq 0} \sum_{|\vec{P}| = 2^{-k}} f_1  * \eta_{P_1} f_2 * \eta_{P_2}  \right| \right|_{\frac{p_1 p_2}{p_1 + p_2}} \simeq \left| \left|\left(  \sum_{k \geq 0} \left| \sum_{|\vec{P}| = 2^{-k}} f_1  * \eta_{P_1} f_2 * \eta_{P_2} \right|^2 \right)^{1/2} \right| \right|_{\frac{p_1 p_2}{p_1 + p_2}}.
\end{eqnarray*}
 At this stage, it is helpful to introduce mollify the frequency projections. Motivation for this approach is mainly technical and should become clearer in the course of the proof. By introducing more smoothness in frequency we expect to extract more decay in time. 
 
 Consider the periodic tent function $T$ defined to be $1-2|x|$ for $|x| \leq 1/2$ and $T(x)=T(y)$ for any $x, y \in \mathbb{R}$ such that $x-y\in \mathbb{Z}$, i.e. T is the periodic extension with period 1 of the tent function to all of $\mathbb{R}$. Also, construct the shifted periodic  tent function $\tilde{T}:= T(x+1/2)$. It is simple matter to see $T(x) + \tilde{T}(x) = 1$ for all $x \in \mathbb{R}$. Define for any $\vec{Q} = (Q_1, Q_2)$

\begin{eqnarray*}
a_{\vec{Q}}&=& T(c_{Q_1}-|Q|/2) \\ 
b_{\vec{Q}} &=& \tilde{T}(c_{Q_1}-|Q|/2).
\end{eqnarray*} 
By construction, $a_{\vec{Q}} + b_{\vec{Q}} =1~ \forall \vec{Q}$. Because of the triangle inequality, it suffices to prove separate estimates for

\begin{eqnarray*}
&&\left| \left|\left(  \sum_{k \geq 0} \left| \sum_{|\vec{P}| = 2^{-k}} a_{\vec{P}} f_1  * \mathcal{F}^{-1} \left[ \hat{\eta}_{P_1} \right]  f_2 * \mathcal{F}^{-1} \left[ \hat{\eta}_{P_2}\right]  \right|^2 \right)^{1/2} \right| \right|_{\frac{p_1 p_2}{p_1 + p_2}} \\ \text{and}&&\left| \left|\left(  \sum_{k \geq 0} \left| \sum_{|\vec{P}| = 2^{-k}} b_{\vec{P}} f_1  * \mathcal{F}^{-1} \left[ \hat{\eta}_{P_1}\right]  f_2 *\mathcal{F}^{-1} \left[ \hat{ \eta}_{P_2} \right]\right|^2 \right)^{1/2} \right| \right|_{\frac{p_1 p_2}{p_1 + p_2}} .
\end{eqnarray*}
So, consider the $b_{\vec{P}}$ part of the sum. Fix $\bar{1}^\prime \in C^{\infty}_{[-1/2, 3/2]}$ satisfying $1_{[0,1]} \leq \bar{1}^\prime \leq 1_{[-1/2, 3/2]}$ and set $\bar{1}^\prime_{I_m} (x):= \bar{1}^\prime(x-m)$.  Compute for $c_{P_1} = m +(\gamma+1/2)2^{-k}$ and $0 \leq \gamma <2^{k}$ using Fourier series on the interval $[m-1/2, m+3/2]$

\begin{eqnarray*}
 \mathcal{F} \left( \left(\bar{1}_{I_0}(\cdot-n) e^{2 \pi i \frac{2}{3} \mu (\cdot -n)}\right) * \eta_{P_1} \right) (\xi)&=& \hat{\bar{1}} _{I_0} \left(\xi- \frac{2}{3} \mu\right)e ^{-2 \pi i n \xi} \cdot\hat{\eta}_{P_1}(\xi) \\ &=&  \hat{\bar{1}} _{I_0} \left(\xi- \frac{2}{3} \mu\right) \bar{1}^\prime_{I_{m}}(\xi) e ^{-2 \pi i n \xi} \cdot\hat{\eta}_{P_1}(\xi)  \\ &=& \sum_{\lambda \in \mathbb{Z}} d^\lambda_{m, \mu} e^{2 \pi i \frac{1}{2} \lambda \xi} \bar{1}^\prime_{I_{m}}(\xi) e ^{-2 \pi i n \xi} \cdot\hat{\eta}_{P_1}(\xi) \\ &=& \sum_{\lambda \in \mathbb{Z}} d^\lambda_{m, \mu} e^{2 \pi i (\frac{ \lambda}{2} -n)\xi} \cdot\hat{\eta}_{P_1}(\xi) ,
\end{eqnarray*}
where $|d^\lambda_{m , \mu}| \lesssim \frac{1}{(1+|\lambda|^N)(1+ | m - \frac{2}{3} \mu |^N)}$. It follows immediately that

\begin{eqnarray*}
\left[ \tilde{1}_{I_0}(\cdot-n) e^{2 \pi i \frac{2}{3} \mu (\cdot -n)}\right] * \eta_{P_1} &=& \sum_{\lambda \in \mathbb{Z}} d^\lambda_{m, \mu}  \cdot \eta_{P_1}\left(x-n+\frac{\lambda}{2}\right) \\&=& \sum_{\lambda \in \mathbb{Z}} d^\lambda_{m, \mu} 2^{-k}  \eta\left(\left(x-n+\frac{\lambda}{2}\right)2^{-k}\right) e^{2 \pi i (m + ( \gamma+1/2) 2^{-k})(x-n + \frac{\lambda}{2})},
\end{eqnarray*}
Putting it all together yields

\begin{eqnarray*}
f_1 *\mathcal{F}^{-1} \left[ \hat{\eta}_{P_1}\right] (x) &=& \sum_{n \in \mathbb{Z}} \sum_{\mu \in \mathbb{Z}} \sum_{\lambda \in \mathbb{Z}} c_{1,n } ^\mu d^\lambda_{m, \mu} 2^{-k}  \eta \left(\left(x-n+\frac{\lambda}{2}\right)2^{-k}\right) e^{2 \pi i (m + (\gamma+1/2) 2^{-k})(x-n + \frac{\lambda}{2})} \\
f_2*\mathcal{F}^{-1} \left[ \eta_{P_2}\right] (x) &=& \sum_{n \in \mathbb{Z}} \sum_{\mu \in \mathbb{Z}} \sum_{\lambda \in \mathbb{Z}} c_{2,n } ^\mu d^\lambda_{m, \mu} 2^{-k}  \eta\left(\left(x-n+\frac{\lambda}{2}\right)2^{-k}\right) e^{2 \pi i (-m - (\gamma+1/2) 2^{-k}+\Gamma 2^{-k})(x-n + \frac{\lambda}{2})}.
\end{eqnarray*}
Furthermore, using $\tilde{1}_{2^k I_{\theta}}(x) := \eta((x-\theta)2^{-k})$,
\begin{eqnarray*}
&& \left|  \sum_{|\vec{P}| = 2^{-k}}b_{\vec{P}} f_1 * \mathcal{F}^{-1} \left[ \hat{\eta}_{P_1}\right] (x) f_2 *\mathcal{F}^{-1} \left[ \hat{ \eta}_{P_2}\right] (x)\right|  \\&=&\left| \sum_{ m \in \mathbb{Z}} \sum_{0 \leq \gamma < 2^{k}} \sum_{n_1, n_2, \mu_1, \mu_2, \lambda _1, \lambda _2 \in \mathbb{Z}}2^{-2k}c_{1, n_1}^{\mu_1} c_{2, n_2}^{\mu_2} d^{\lambda_1}_{m, \mu_1} d^{\lambda _2} _{m, \mu_2}   \tilde{1}_{2^k I_{n_1-\frac{\lambda_1}{2}}}(x) \tilde{1}_{2^k I_{n_2 - \frac{\lambda_2}{2}}}(x)\right. \\&\times&\left.  \tilde{T}(m+\gamma2^{-k}) e^{2 \pi i (m+  (\gamma+1/2) 2^{-k})(n_2-n_1 + \frac{(\lambda_1 - \lambda _2)}{2})} e^{ 2 \pi i \Gamma 2^{-k} (n_2-\frac{\lambda_2}{2})}\right|\\ &=&  \left| \sum_{ m \in \mathbb{Z}}  \sum_{n_1, n_2, \mu_1, \mu_2, \lambda _1, \lambda _2 \in \mathbb{Z}}2^{-2k}c_{1, n_1}^{\mu_1} c_{2, n_2}^{\mu_2} d^{\lambda_1}_{m, \mu_1} d^{\lambda _2} _{m, \mu_2}   \tilde{1}_{2^k I_{n_1-\frac{\lambda_1}{2}}}(x) \tilde{1}_{2^k I_{n_2 - \frac{\lambda_2}{2}}}(x)  \right. \\&\times&\left.  \left[ \sum_{0 \leq \gamma < 2^{k}}\tilde{T}(m+\gamma2^{-k}) e^{2 \pi i   \gamma 2^{-k}(n_2-n_1 + \frac{(\lambda_1 - \lambda _2)}{2})} \right]e^{2 \pi i (m+2^{-k-1})(n_2-n_1 +\frac{(\lambda _1 - \lambda_2)}{2})}e^{ -2 \pi i \Gamma 2^{-k} (n_2-\frac{\lambda_2}{2})}\right|.
\end{eqnarray*}
The sum over $\gamma$ is

\begin{eqnarray*}
\sum_{0 \leq \gamma <2^k} \tilde{T}(m+\gamma 2^{-k}) e^{ 2 \pi i \gamma 2^{-k} (n_2 -n_1 + \frac{(\lambda _1 - \lambda_2)}{2})} &=& \sum_{0 \leq \gamma <2^k} \tilde{T}(\gamma 2^{-k}) e^{ 2 \pi i \gamma 2^{-k} (n_2 -n_1 + \frac{(\lambda _1 - \lambda_2)}{2})} \\ &=&  \sum_{-2^{k-1} \leq \gamma <2^{k-1}} 
(1-2|\gamma 2^{-k}|) e^{ 2 \pi i \gamma 2^{-k} (n_2 -n_1 + \frac{(\lambda _1 - \lambda_2)}{2})} e^{2 \pi i \frac{1}{2} (n_2-n_1 + \frac{\lambda_1 - \lambda_2}{2})}\\ &=& F_{2^k} \left(2^{-k} \left(n_2-n_1 +\frac{(\lambda_1 - \lambda_2)}{2}\right)\right)  e^{2 \pi i \frac{1}{2} (n_2-n_1 + \frac{\lambda_1 - \lambda_2}{2})}.
\end{eqnarray*}
Therefore, 

\begin{eqnarray*}
&&\left|  \sum_{|\vec{P}| = 2^{-k}}b_{\vec{P}} f_1 *\mathcal{F}^{-1} \left[  \hat{\eta}_{P_1}\right](x) f_2 * \mathcal{F}^{-1} \left[ \hat{\eta}_{P_2}\right](x) \right| \\ &=&\left|   \sum_{ m \in \mathbb{Z}}  \sum_{n_1, n_2, \mu_1, \mu_2, \lambda _1, \lambda _2 \in \mathbb{Z}}2^{-2k}c_{1, n_1}^{\mu_1} c_{2, n_2}^{\mu_2} d^{\lambda_1}_{m, \mu_1} d^{\lambda _2} _{m, \mu_2}   \tilde{1}_{2^k I_{n_1-\frac{\lambda_1}{2}}}(x) \tilde{1}_{2^k I_{n_2 - \frac{\lambda_2}{2}}}(x)  \right. \\&\times& \left. F_{2^k} \left(2^{-k} (n_2-n_1 +\frac{(\lambda_1 - \lambda_2)}{2})\right) e^{2 \pi i (m+\frac{1}{2}+2^{-k-1})(n_2-n_1 +\frac{(\lambda _1 - \lambda_2)}{2})}e^{- 2 \pi i \Gamma 2^{-k} (n_2-\frac{\lambda_2}{2})}\right|.
\end{eqnarray*}
The sum containing coefficients $a_{\vec{P}}$ is handled similarly as the sum containing coefficients $b_{\vec{P}}$, so the details are omitted. The expression one derives in this case is 

\begin{eqnarray*}
&& \left| \sum_{|\vec{P}|=2^{-k}} a_{\vec{P}} f_1*\mathcal{F}^{-1} \left[ \hat{\eta}_{P_1}\right] (x) f_2*\mathcal{F}^{-1} \left[ \eta_{P_2}\right] (x) \right| \\ &=& \left|  \sum_{ m \in \mathbb{Z}}  \sum_{n_1, n_2, \mu_1, \mu_2, \lambda _1, \lambda _2 \in \mathbb{Z}}2^{-2k}c_{1, n_1}^{\mu_1} c_{2, n_2}^{\mu_2} \tilde{d}^{\lambda_1}_{m, \mu_1} \tilde{d}^{\lambda _2} _{m, \mu_2}   \tilde{1}_{2^k I_{n_1-\frac{\lambda_1}{2}}}(x) \tilde{1}_{2^k I_{n_2 - \frac{\lambda_2}{2}}}(x)  \right. \\&\times& \left. F_{2^k} \left(2^{-k} (n_2-n_1 +\frac{(\lambda_1 - \lambda_2)}{2})\right) e^{2 \pi i (m+2^{-k-1})(n_2-n_1 +\frac{(\lambda _1 - \lambda_2)}{2})}e^{ -2 \pi i \Gamma 2^{-k} (n_2-\frac{\lambda_2}{2})}\right|,
\end{eqnarray*}
 where $\tilde{d}^\lambda_{m, \mu}$ satisfies the same decay properties as $d^\lambda_{m, \mu}$, i.e. $|\tilde{d}^\lambda_{m, \mu}| \lesssim_N \frac{1}{(1+|\lambda|^N)(1+|m-\frac{2\mu}{3}|^N)}$. Now it suffices to make use of the point-wise bound for the Fej\'er kernel: 

\begin{eqnarray*}
\left| F_{2^k} \left( 2^{-k} \left(n_2-n_1 + \frac{(\lambda_1 - \lambda _2)}{2}\right)\right)\right| \lesssim \frac{2^k}{1 + \left| n_2-n_1 + \frac{(\lambda_1 - \lambda _2)}{2}\right|^2}
\end{eqnarray*}
assuming $|n_2-n_1 + \frac{\lambda_1 - \lambda _2}{2}| \lesssim 2^k$. By $2^k$-periodicity, we have the same bound for $|n_2-n_1 + \frac{ \lambda_1 - \lambda _2}{2}| \simeq \kappa 2^k$:

\begin{eqnarray*}
\left| F_{2^k} \left( 2^{-k} \left(n_2-n_1 + \frac{(\lambda_1 - \lambda _2)}{2}\right)\right)\right| \lesssim \frac{2^k}{1 + \left| n_2-n_1 + \frac{(\lambda_1 - \lambda _2)}{2}-\kappa2^k \right|^2}.
\end{eqnarray*}

\subsection{Dominant Contribution}
The dominant contribution to the sum of  frequency projections over a given scale arises from the terms corresponding to $\frac{2}{3}\mu_1 \simeq \frac{2}{3} \mu_2 \simeq m$, along with $\lambda _1 = \lambda _2 =0$ and $\kappa=0$. Under these assumptions, the sum over $m,n_1, n_2: |n_1-n_2| <2^{k-1}$ is majorized by

\begin{eqnarray*}
 2^{-2k}  d^0_{m,m} d^0_{m,m}c_{1, n_1}^{m}\tilde{1}_{2^kI_{n_1}}(x)  \cdot c_{2, n_2}^{m}  \tilde{1}_{2^k I_{n_2 }}(x) e^{ 2 \pi i (m+ 2^{-k-1} )\left(n_2 -n_1 \right)} e^{-2 \pi i \Gamma 2^{-k} n_2} F_{2^k} \left(2^{-k} (n_1-n_2)\right).
\end{eqnarray*}	
Furthermore, if $n_1=n_2$, we face

\begin{eqnarray*}
&& \left| \sum_{ m, n_1 \in \mathbb{Z}}  2^{-k}d^0_{m,m} d^0_{m,m} c_{1, n_1}^{m}c_{2, n_1}^{m} \tilde{1}_{2^kI_{n_1}} (x)   \tilde{1}_{2^k I_{n_1 }}(x) e^{-2 \pi i \Gamma 2^{-k} n_1}\right| =~~\left| \left[ \sum_{ m,n_1 \in \mathbb{Z}} \left[ d^0_{m,m} \right]^2 c_{1, n_1}^{m}c_{2, n_1}^{m} \tilde{\tilde{1}}_{I_{n_1}}\right] *\tilde{\psi}_k(x) \right|,
\end{eqnarray*}
where $\tilde{\tilde{1}}_{I}(x):= \tilde{\tilde{1}}(x-c_I)$,  $\tilde{\tilde{1}} \in \mathcal{S}(\mathbb{R})$ satisfies $1_{[0,1]} \leq  \mathcal{F} \left( \tilde{\tilde{1}}\right) \leq 1_{[-1, 2]}$, and $\tilde{\psi}_k(x):= 2^{-k} \tilde{\psi}(2^{-k} x)$ where $\tilde{\psi}(x):=\tilde{1}^2(x) e ^{2 \pi i \Gamma x}$. 
Therefore, we are able to compute the $L^p$ norm of this dominant expression as

\begin{eqnarray*}
&& \left|\left| \left(\sum_{k \geq 0} \left| \left[ \sum_{ m,n_1 \in \mathbb{Z}} \left[ d^0_{m,m} \right]^2  c_{1, n_1}^{m}c_{2, n_1}^{m}\tilde{\tilde{1}}_{I_{n_1}}\right] *\tilde{\psi}_k \right|^2 \right)^{1/2} \right|\right|_\frac{p_1 p_2}{p_1 + p_2}\\& \lesssim& \left|\left|\sum_{ m,n_1 \in \mathbb{Z}} \left[ d^0_{m,m} \right]^2 c_{1, n_1}^{m}c_{2, n_1}^{m}\tilde{\tilde{1}}_{I_{n_1}} \right|\right|_\frac{p_1 p_2}{p_1 + p_2} \\ &\leq& \left|\left| \sum_{n_1\in \mathbb{Z}} \left(\sum_{m \in \mathbb{Z}}| c_{1, n_1}^{m}|^2 \right)^{1/2} \left( \sum_{ m \in \mathbb{Z}} |c_{2, n_1}^{m}|^2 \right)^{1/2} \left| \tilde{\tilde{1}}_{I_{n_1}} \right|\right|\right|_\frac{p_1 p_2}{p_1 + p_2} \\ &\leq&  \left|\left|\left(  \sum_{n_1\in \mathbb{Z}} \left(\sum_{m \in \mathbb{Z}}| c_{1, n_1}^{m}|^2 \right)^{1/2} \left|  \tilde{\tilde{1}}_{I_{n_1}} \right|\right)\left( \sum_{n_2 \in \mathbb{Z}} \left( \sum_{ m \in \mathbb{Z}} |c_{2, n_1}^{m}|^2 \right)^{1/2} \left| \tilde{\tilde{1}}_{I_{n_2}}\right| \right|\right|\right|_\frac{p_1 p_2}{p_1 + p_2} \\ &\leq& \left|\left|  \sum_{n_1\in \mathbb{Z}} \left(\sum_{m \in \mathbb{Z}}| c_{1, n_1}^{m}|^2 \right)^{1/2} \left|  \tilde{\tilde{1}}_{I_{n_1}}\right|  \right| \right|_{p_1} \left|  \left| \sum_{n_2 \in \mathbb{Z}} \left( \sum_{ m \in \mathbb{Z}} |c_{2, n_1}^{m}|^2 \right)^{1/2} \left| \tilde{\tilde{1}}_{I_{n_2}} \right| \right|\right|_{p_2} \\  &\simeq& \left( \sum_{n_1 \in \mathbb{Z}} \left( \sum_{m \in \mathbb{Z}} |c^m_{1, n_1}|^2 \right)^{p_1/2} \right)^{1/p_1} \left( \sum_{n_2 \in \mathbb{Z}} \left( \sum_{m \in \mathbb{Z}} |c^m_{2, n_2}|^2 \right)^{p_2/2} \right)^{1/p_2} \\ &\leq& \left( \sum_{n_1 \in \mathbb{Z}}  || f_1 \bar{1}_{I_n}||_{p_1}^{p_1} \right)^{1/p_1} \left( \sum_{n_2 \in \mathbb{Z}} ||f_2\bar{1}_{I_{n_2}}||_{p_2}^{p_2} \right)^{1/p_2} \\  &\lesssim& || f_1||_{p_1} ||f_2||_{p_2}. 
\end{eqnarray*}
The assumption $2 \leq p_1, p_2 <\infty$ is necessary to observe $\left( \sum_{m_1 \in \mathbb{Z}} |c_{1, n_1}^{m_1}|^2 \right)^{1/2} \simeq || f_1 \bar{1}_{I_{n_1}} ||_2 \lesssim ||f_1 \bar{1}_{I_{n_1}}||_{p_1}$.

Before removing restrictions on the parameters $\mu_1, \mu_2, \lambda_1, \lambda_2, \kappa, n_1, n_2$, it is useful to observe
\begin{lemma}\label{ML}
Let $\tilde{1}\in \mathcal{S}(\mathbb{R})$ have compact Fourier support inside $[-1,1]$. Then, there exists $\tilde{\tilde{1}}\in \mathcal{S}(\mathbb{R})$ such that for all $L \in \mathbb{R}$

\begin{eqnarray*}
\tilde{1}_{2^kI_{\theta}} (x) \tilde{1}_{2^k I_{\theta +L}} (x) e^{2 \pi i \Gamma 2^{-k} (x-\theta-L)} = \frac{2^k}{1+(2^{-k}L)^2} \left(Tr_{\theta +L} \tilde{\tilde{1}} \right)* \psi_k^{L}(x)
\end{eqnarray*}
where $\psi_k ^{L}$ is lacunary at scale $2^k$ and uniformly Mikhlin in the parameter $L$.  
\end{lemma}

\begin{proof}
The result follows from a direct application of the Fourier transform:
\begin{eqnarray*}
2^{-k} \mathcal{F} (LHS) (\xi)&=& 2^{-k} \mathcal{F} \left[ \tilde{1}_{2^kI_{\theta}} (\cdot) \tilde{1}_{2^k I_{\theta +L}}(\cdot) e^{-2 \pi i \Gamma 2^{-k} (\cdot-L-\theta)}  \right] (\xi)\\ &=& 2^{-k} \mathcal{F} \left[ \tilde{1}_{2^k I_0} (\cdot +L) \tilde{1}_{2^k I_0}(\cdot ) e^{2 \pi i \Gamma 2^{-k} \cdot }\right] (\xi) e^{-2 \pi i \xi (L+\theta)} \\ &=& 2^{-k} \mathcal{F} \left[ \tilde{1}_{ I_0} (2^{-k} \cdot +2^{-k} L) \tilde{1}_{I_0}(2^{-k} \cdot) e^{2 \pi i \Gamma 2^{-k} \cdot}\right](\xi)  e^{-2 \pi i \xi (L+\theta)} \\ &=& \mathcal{F} \left[ \tilde{1}_{ I_0} (\cdot +2^{-k}L) \tilde{1}_{I_0}( \cdot) e^{2 \pi i \Gamma \cdot}\right] (2^k \xi) e^{-2 \pi i \xi (L+\theta)}\\ &=& \left[ \left( \hat{\tilde{1}}_{ I_0} (\cdot )e^{2 \pi i 2^{-k} L \cdot}\right) * \hat{ \tilde{1}}_{I_0}( \cdot)\right] (2^k\xi-\Gamma) e^{-2 \pi i \xi (L+\theta)} .
\end{eqnarray*}
The support of $\left[ \hat{\tilde{1}}_{ I_0} (\cdot )e^{2 \pi i 2^{-k} L \cdot} * \hat{ \tilde{1}}_{I_0}( \cdot)\right] (2^k\xi-\Gamma)$ is contained inside $[0, K]$ for $K \simeq 1$. Therefore, one may insert another function  $\Phi \in \mathcal{S}(\mathbb{R})$ (which is identically equal to one on $[0,K]$) into the last expression:

\begin{eqnarray*}
2^{-k}\mathcal{F}(LHS)(\xi)&=&\left[  \Phi (\xi) e^{-2 \pi i \xi (L+\theta)}\right]\left[ \hat{\tilde{1}}_{ I_0} (\cdot )e^{2 \pi i 2^{-k}L \cdot} * \hat{ \tilde{1}}_{I_0}( \cdot))\right] (2^k\xi-\Gamma) \\ &:=& \frac{1}{1+\kappa^2}\left[  \Phi (\xi) e^{-2 \pi i \xi (L+\theta)}\right] \hat{\psi}^{L}_k(\xi).
\end{eqnarray*}
Hence, 
\begin{eqnarray*}
LHS = \frac{2^k}{1+(L2^{-k})^2}\left( Tr_{\theta+L} \check{\Phi}\right) * \psi^{L}_k=RHS,
\end{eqnarray*}
provided we set $\tilde{\tilde{1}}(x) = \check{\Phi}(x)$. 
\end{proof}

\subsection{Removing Restrictions}
Throwing in those pairs $(n_1, n_2)$ for which  $n_1 \not = n_2$, we may bound using Lemma \ref{ML}

\begin{eqnarray*}
&&\sum_{ m \in \mathbb{Z}}~~\sum_{n_1\in \mathbb{Z}} \sum_{\kappa \in \mathbb{Z}} ~\sum_{l: |l- \kappa 2^k| \lesssim 2^k} 2^{-2k}c_{1, n_1}^{m} c_{2, n_1+l}^{m}\tilde{1}_{2^kI_{n_1}} (x)  \tilde{1}_{2^k I_{n_1+l }}(x) e^{2 \pi i 2^{-k-1}l} e^{2 \pi i \Gamma 2^{-k} (n_1+l)} F_{2^k} \left(2^{-k} l\right) \\ &=& \sum_{\kappa \in \mathbb{Z}} ~\sum_{ l:|l-\kappa 2^k| \lesssim 2^k}2^{-k}F_{2^k} \left(2^{-k} l\right) e^{2 \pi i 2^{-k-1}l} \left[\sum_{ m \in \mathbb{Z}}~~\sum_{n_1\in \mathbb{Z}}  2^{-k} c_{1, n_1}^{m} c_{2, n_1+l}^{m}\tilde{1}_{2^kI_{n_1}}   \tilde{1}_{2^k I_{n_1+l }} e^{2 \pi i \Gamma 2^{-k} (n_1+l)} \right] \\ &=&\sum_{\kappa \in \mathbb{Z}}~ \sum_{l: |l-\kappa 2^k| \leq 2^{k-1}} \frac{1}{\langle \kappa \rangle^2 \langle l \rangle^2}\left[\sum_{ m \in \mathbb{Z}}~~\sum_{n_1\in \mathbb{Z}}   c_{1, n_1}^{m} c_{2, n_1+l}^{m} \tilde{\tilde{1}}_{I_{n_1+l}}\right]*\psi^l_k.
\end{eqnarray*}
Hence, 
\begin{eqnarray*}
&& \left| \left| \left( \sum_{k \in \mathbb{Z}} \left| \sum_{ m \in \mathbb{Z}}~~\sum_{n_1\in \mathbb{Z}} \sum_{\kappa \in \mathbb{Z}} \sum_{l : | l - \kappa 2^k| \leq 2^{k-1}} 2^{-2k} c_{1, n_1}^{m} c_{2, n_2}^{m}\tilde{1}_{2^kI_{n_1}}   \tilde{1}_{2^k I_{n_1+l }}e^{2 \pi i 2^{-k-1}l}e^{2 \pi i \Gamma 2^{-k} (n_1 +l)} F_{2^k} \left(2^{-k} l\right) \right|^2 \right)^{1/2}\right| \right|_{\frac{p_1 p_2}{p_1 + p_2}} \\ &\lesssim&  \sum_{\kappa, l \in \mathbb{Z}} \frac{1}{\langle \kappa \rangle^2 \langle l \rangle ^2} \left| \left| \left( \sum_{k \in \mathbb{Z}} \left| \left[\sum_{ m \in \mathbb{Z}}~~\sum_{n_1\in \mathbb{Z}}   c_{1, n_1}^{m} c_{2, n_1+l}^{m} \tilde{1}_{I_{n_1+l}} \right] *\psi^l_k\right|^2 \right)^{1/2} \right| \right|_{\frac{p_1 p_2}{p_1 + p_2}} \\ &\lesssim&  \sup_{l \in \mathbb{Z}} \left| \left| \left( \sum_{k \in \mathbb{Z}} \left|  \left[\sum_{ m \in \mathbb{Z}}~~\sum_{n_1\in \mathbb{Z}}   c_{1, n_1}^{m} c_{2, n_1+l}^{m} \tilde{1}_{I_{n_1+l}}*\psi^l_k\right] \right|^2 \right)^{1/2} \right| \right|_{\frac{p_1 p_2}{p_1 + p_2}} .
\end{eqnarray*}
It therefore suffices to prove estimates for $\left( \sum_{k \in \mathbb{Z}} \left|  \left[\sum_{ m \in \mathbb{Z}}~~\sum_{n_1\in \mathbb{Z}}   c_{1, n_1}^{m} c_{2, n_1+l}^{m}\tilde{1}_{I_{n_1}}*\psi^l_k\right] \right|^2 \right)^{1/2}$ that are independent of $l \in \mathbb{Z}$. However, this is immediate once we use the boundedness of the square function $S:L^p \rightarrow L^p$ for every $1 <p <\infty$ and each $l \in \mathbb{Z}$ together the fact that $\{\psi_k^l\}$ are uniformly Mikhlin.

Because of the rapid decay in the remaining parameters $\mu_1, \mu_2, \lambda_1, \lambda_2, \kappa$, one expects to use triangle inequality to sum over these parameters outside the $L^p$ norm and reducing matters to proving uniform estimates inside the $L^p$ norm.  This is indeed the case, as we now show.  Changing variables in $\mu_1, \mu_2$, and using the triangle inequality, it suffices to estimate the sum over $k, m, n_1, \Delta_1, \Delta_2, \kappa$ and $l:|l-\kappa2^k|\leq  2^{k-1}$ of

\begin{eqnarray*}
 2^{-2k} d^{\lambda_1} _{m, m+\Delta_1} d^{\lambda_2}_{m, m+\Delta_2} c_{1, n_1}^{m+\Delta_1}c_{2, n_1+\tilde{l} + \frac{\lambda_2-\lambda_1}{2} }^{m+\Delta_2}\tilde{1}_{2^kI_{n_1-\frac{\lambda_1}{2}}}(x)    \tilde{1}_{2^k I_{n_1+\tilde{l} - \frac{\lambda_1}{2} } }(x) e^{2 \pi i \Gamma 2^{-k} (n_1 +\tilde{l} +\frac{\lambda_2-\lambda_1}{2})} F_{2^k} \left(2^{-k} \tilde{l}\right).
\end{eqnarray*}
Define $\tilde{l} := n_2 -n_1 + \frac{(\lambda_1 - \lambda_2)}{2} $. Cheaply bringing the summations over $\Delta_1, \Delta_2, \lambda_1, \lambda_2$ outside the $L^p$ norm, using Lemma \ref{ML}, note that the sum in question is majorized by
\begin{eqnarray*}
\sum_{\Delta_1, \Delta_2, \lambda_1, \lambda_2, \kappa, \tilde{l}}  \frac{\left| \left|  \left( \sum_{k \in \mathbb{Z}} \left|  \left[\sum_{ m \in \mathbb{Z}}~~\sum_{n_1\in \mathbb{Z}}  \tilde{d}^{\lambda_1}_{m, m+\Delta_1} \tilde{d}^{\lambda_2}_{m,m+\Delta_2}  c_{1, n_1}^{m+\Delta_1} c_{2, n_1+\tilde{l} +\frac{\lambda_2-\lambda_1}{2}}^{m+\Delta_2} \tilde{\tilde{1}}_{I_{n_2-\frac{\lambda_2}{2}}}*\psi^{\tilde{l}, \kappa, \lambda_1, \lambda_2}_k\right] \right|^2 \right)^{1/2} \right| \right|_{\frac{p_1 p_2}{p_1 + p_2}}}{\langle l \rangle ^2 \langle \kappa \rangle^2 \langle \Delta_1 \rangle ^N \langle \Delta_2 \rangle ^N \langle \lambda_1 \rangle ^N \langle \lambda_2 \rangle^N}  ,
\end{eqnarray*}
 where $\{ \psi_k^{\tilde{l}, \kappa, \lambda_1, \lambda_2} \}$ forms a Littlewood-Paley decomposition and is uniformly Mikhlin  in the sense that
\begin{eqnarray*}
\left| \left| \frac{d^N}{d\xi^N} \mathcal{F}\left(\psi_k^{\tilde{l}, \kappa, \lambda_1, \lambda_2} \right) \right| \right|_{L^\infty(\mathbb{R})} \lesssim 2^{kN}
\end{eqnarray*}
for all $k  \geq 0$ and sufficiently many derivatives. 
It therefore suffices to estimate
\begin{eqnarray*}
&& \left| \left|  \left( \sum_{k \in \mathbb{Z}} \left|  \left[\sum_{ m \in \mathbb{Z}}~~\sum_{n_1\in \mathbb{Z}}  \tilde{d}^{\lambda_1}_{m, m+\Delta_1} \tilde{d}^{\lambda_2}_{m,m+\Delta_2}  c_{1, n_1}^{m+\Delta_1} c_{2, n_1+\tilde{l} +\frac{\lambda_2-\lambda_1}{2}}^{m+\Delta_2} \tilde{\tilde{1}}_{I_{n_2-\frac{\lambda_2}{2}}}*\psi^{\tilde{l}, \kappa, \lambda_1, \lambda_2}_k\right] \right|^2 \right)^{1/2} \right| \right|_{\frac{p_1 p_2}{p_1 + p_2}}
\end{eqnarray*}
with a bound independent of the parameters $\Delta_1, \Delta_2, \lambda_1, \lambda_2, \kappa$, and $\tilde{l}$. 

Since $\frac{1}{p_1} + \frac{1}{p_2} <1$, observe
\begin{eqnarray*}
&&\left| \left|   \sum_{ m \in \mathbb{Z}}~~\sum_{n_1\in \mathbb{Z}}  \tilde{d}^{\lambda_1}_{m, m+\Delta_1} \tilde{d}^{\lambda_2}_{m, m+\Delta_2} c_{1, n_1}^{m+\Delta_1} c_{2, n_1+\tilde{l} +\frac{\lambda_2-\lambda_1}{2}}^{m+\Delta_2} \tilde{\tilde{1}}_{I_{n_2-\frac{\lambda_2}{2}}} \right| \right|_{\frac{p_1 p_2}{p_1 + p_2}}  \\ &\lesssim& \left| \left|   \left( \sum_{n_1\in \mathbb{Z}}\sum_{ m_1 \in \mathbb{Z}}   |c_{1, n_1}^{m_1}|^2\left| \tilde{ \tilde{1}}_{I_{n_2-\frac{\lambda_2}{2}}}\right| \right)^{1/2} \left( \sum_{n_2 \in \mathbb{Z}}\sum_{m_2 \in \mathbb{Z}} \left|c_{2, n_2+\tilde{l} + \frac{\lambda_2 - \lambda_1}{2}}^{m_2}\right|^2 \left| \tilde{\tilde{1}}_{I_{n_2-\frac{\lambda_2}{2}}}\right|  \right)^{1/2}\right| \right|_{\frac{p_1 p_2}{p_1 + p_2}}   \\ &\lesssim& \left| \left|  \left[ \sum_{n_1\in \mathbb{Z}}\left( \sum_{ m_1 \in \mathbb{Z}}   |c_{1, n_1}^{m_1}|^2 \right)^{1/2} \left| \tilde{\tilde{1}}_{I_{n_1-\frac{\lambda_2}{2}}}  \right|\right] \left[ \sum_{n_2 \in \mathbb{Z}}  \left( \sum_{m_2 \in \mathbb{Z}} \left|c_{2, n_2+\tilde{l} + \frac{ \lambda_2 - \lambda_2}{2}}^{m_2}\right|^2  \right)^{1/2} \left| \tilde{\tilde{1}}_{I_{n_2-\frac{\lambda_2}{2}}} \right|\right] \right| \right|_{\frac{p_1 p_2}{p_1 + p_2}}  \\ &\lesssim& || f_1||_{p_1} || f_2||_{p_2}.
\end{eqnarray*}
Like before, the assumption $2 \leq p_1, p_2 <\infty$ is necessary to achieve the last inequality. 
\end{proof}

\small
\nocite{*}
\bibliography{MVC}{}
\bibliographystyle{plain}

 \end{document}